\documentclass{amsart}

\usepackage[colorlinks=true, urlcolor=black, citecolor=black, linkcolor=black, hyperfootnotes=true]{hyperref}

\usepackage[latin1]{inputenc}
\usepackage{amsfonts}
\usepackage{amsmath}
\usepackage{amsthm}
\usepackage{amssymb}
\usepackage{latexsym}
\usepackage{enumerate}
\usepackage{tikz-cd}

\newtheorem{theorem}{Theorem}[section]
\newtheorem{lemma}[theorem]{Lemma}
\newtheorem{proposition}[theorem]{Proposition}
\newtheorem{corollary}[theorem]{Corollary}

\newtheorem{remark}[theorem]{Remark}

\newtheorem{question}[theorem]{Question}

\newtheorem{definition}[theorem]{Definition}

\numberwithin{equation}{section}

\begin{document}

\newcommand{\cc}{\mathfrak{c}}
\newcommand{\N}{\mathbb{N}}
\newcommand{\BB}{\mathbb{B}}
\newcommand{\C}{\mathbb{C}}
\newcommand{\Q}{\mathbb{Q}}
\newcommand{\R}{\mathbb{R}}
\newcommand{\T}{\mathbb{T}}
\newcommand{\st}{*}
\newcommand{\PP}{\mathbb{P}}
\newcommand{\QQ}{\mathbb{Q}}
\newcommand{\GG}{\mathbb{G}}
\newcommand{\EE}{\mathbb{E}}
\newcommand{\FF}{\mathbb{F}}
\newcommand{\DD}{\mathbb{D}}
\newcommand{\lin}{\left\langle}
\newcommand{\rin}{\right\rangle}
\newcommand{\SSS}{\mathbb{S}}
\newcommand{\forces}{\Vdash}
\newcommand{\dom}{\text{dom}}
\newcommand{\osc}{\text{osc}}
\newcommand{\F}{\mathcal{F}}
\newcommand{\A}{\mathcal{A}}
\newcommand{\B}{\mathcal{B}}
\newcommand{\CC}{\mathcal{C}}
\newcommand{\I}{\mathcal{I}}
\newcommand{\J}{\mathcal{J}}

\author{Clayton Suguio Hida}
\address{Departamento de Matem\'atica, Instituto de Matem\'atica e Estat\'\i stica, Universidade
de S\~ao Paulo, Caixa Postal 66281, 05314-970, S\~ao Paulo, Brazil}
\email{\texttt{suguio@ime.usp.br}}
\thanks{The research of the first named author was partially supported by   doctoral scholarships
 CAPES:  1427540 and 
CNPq:    167761/2017-0 and 201213/2016-8.}

\author{Piotr Koszmider}
\address{Institute of Mathematics of the Polish Academy of Sciences,
ul. \'Sniadeckich 8,  00-656 Warszawa, Poland}
\email{\texttt{piotr.koszmider@impan.pl}}
\thanks{The research of the second named author was partially supported by   grant
PVE Ci\^encia sem Fronteiras - CNPq (406239/2013-4).}

\subjclass[2010]{03E35, 46L05, 46L85, 46J10,  54G12}
\title[Large irredundant sets in operator algebras]{
Large irredundant sets in operator algebras}

\begin{abstract} 
A subset $\mathcal X$ of a C*-algebra $\A$ is called irredundant if no $A\in \mathcal X$
belongs to the C*-subalgebra of $\A$ generated by  $\mathcal X\setminus \{A\}$. 
Separable C*-algebras cannot have uncountable irredundant sets and
all members of many classes of  nonseparable C*-algebras, e.g., infinite dimensional von Neumann algebras
 have irredundant sets of cardinality continuum.

There exists a considerable literature 
showing that the question whether every AF commutative nonseparable 
C*-algebra  has an uncountable irredundant set is sensitive to
additional set-theoretic axioms and we investigate here the noncommutative
case.

Assuming $\diamondsuit$ (an additional axiom stronger than the continuum hypothesis) we   prove that  there
is an AF  C*-subalgebra of $\B(\ell_2)$ of density 
$2^\omega=\omega_1$ with no 
nonseparable commutative C*-subalgebra and with no
uncountable irredundant set.
On the other hand we also prove that it is consistent that every discrete collection of operators in $\B(\ell_2)$ 
of cardinality continuum contains an irredundant subcollection of cardinality continuum. 

Other partial results and more open problems are presented.

\end{abstract}

\maketitle

\section{Introduction}

\begin{definition}
Let $\mathcal{A}$ be a C*-algebra. A subset $\mathcal X\subseteq \mathcal{A}$
 is called irredundant  if and only if for every $A \in \mathcal X$, the 
 C*-subalgebra of $\mathcal{A}$ generated by $\mathcal X\setminus\{A\}$ 
 does not contain $A$. 
We define $$irr(\mathcal{A}):=\sup\{|\mathcal X|: \mathcal X\
 \textrm{ is an irredundant set in } \mathcal{A}\}.$$
\end{definition}

Recall that the density of a C*-algebra $\A$, denoted $d(\A)$ is the least cardinality of a 
norm dense subset of $\A$, i.e., $\A$ is separable if and only if
$d(\A)$ is countable.
It is easy to see that $irr(\A)\leq d(\A)$ for every C*-algebra, as irredundant sets must be
norm discrete. When $\A$ is an infinite dimensional C*-algebra, then $irr(\A)$ is infinite, 
because then $\A$ contains an infinite dimensional abelian C*-subalgebra (\cite{ogasawara}) and locally compact
infinite Hausdorff spaces contain pairwise disjoint infinite collections of open sets which yield
infinite irredundant sets (Proposition \ref{orthogonal}).
In this article, we are interested in  uncountable irredundant sets
in  (C*-subalgebras of) the algebra $\B(\ell_2)$ of all linear bounded operators  on a separable Hilbert space.

Irredundant sets have been considered in the context of other structures. For example, 
a subset of a Boolean algebra is called irredundant if none of its elements belongs 
to the Boolean subalgebra generated by the remaining elements. We call
such sets Boolean irredundant (Definition \ref{def-boolean-irr}). In Banach
spaces irredundant sets, i.e., where no element belongs to the closed subspace
spanned by the remaining elements correspond exactly to biorthogonal systems
(\cite{hajek}, see \cite{invariants} for some comparisons between this type of notions). 
Examples of Boolean irredundant sets include independent families, ideal independent families
or (almost) disjoint families, but there are Boolean algebras of uncountable
irredundance with no uncountable families of the above-mentioned classes (see Remark
\ref{remark-split}). A collection $(x_\alpha, x_\alpha^*)_{\alpha<\kappa}\subseteq \B\times\B^*$
of a Banach space $\B$ is biorthogonal if $x_\alpha^*(x_\alpha)=1$ and 
$x_\alpha^*(x_\beta)=0$ for all $\alpha<\beta<\kappa$. As usually linear
functionals on a C*-algebra $\A$ are not multiplicative there are many more biorthogonal
systems than irredundant sets in $\A$, one can even consistently have 
a commutative C*-algebra $C(K)$ with countable irredundance but with 
uncountable biorthogonal systems (\cite{finitely-supp}).

One of our main motivations are consistent constructions of uncountable
Boolean algebras with no uncountable irredundant sets. 
They were first obtained 
by Rubin (\cite{rubin}) under the assumption of $\diamondsuit$\footnote{$\diamondsuit$ is an additional 
axiom (introduced by R. Jensen) which is true in the universe of constructible sets. It says that
there is a sequence $(S_\alpha)_{\alpha<\omega_1}$ which ``predicts" all subsets of $\omega_1$
in the sense that
for any $X\subseteq\omega_1$ the set $\{\alpha<\omega_1: X\cap\alpha=S_\alpha\}$
meets every closed an unbounded subset of $\omega_1$, for details see \cite{jech} or \cite{kunen}.
$\diamondsuit$ has been recently successfully applied in the context 
of nonseparable C*-algebras by Akemann, Farah, Hirshberg and Weaver 
(\cite{akemann-weaver-naimark}, \cite{akemann-weaver-pure}, \cite{farah-hirshberg}).
We will not use the $\diamondsuit$ directly but will apply its consequence from
Theorem \ref{diamond-scheme} which was developed by S. Todorcevic in \cite{stevo-scheme}.} and then by Kunen
(\cite{negrepontis}) under  the continuum hypothesis CH (improved
further by Todorcevic to a
$\mathfrak{b}=\omega_1$ construction from 2.4 of \cite{stevo-partition}). 
Also some versions of the classical Ostaszewski's construction assuming $\diamondsuit$
from \cite{ostaszewski} have these properties as  further constructions assuming
$\clubsuit$ from \cite{hajek} as well as forcing
constructions from \cite{bgt}, \cite{finitely-supp}, \cite{rolewicz}.

Some of the above constructions are of Boolean algebras and other
of (locally) compact Hausdorff totally disconnected spaces. Using the Stone duality
one translates one language to the other easily.
The fact that  the Kunen or Ostaszewski types of constructions mentioned above 
correspond to superatomic Boolean algebras
or equivalently their Stone spaces are 
scattered  spaces (every subset has a relative isolated point)
yields the
equality between  the Boolean irredundance 
of the Boolean algebra and  the irredundance 
of the commutative C*-algebra of continuous functions (Corollary
\ref{cor-boolean}). In particular  the corresponding
$C(K)$s have no uncountable irredundant sets. In fact the scatteredness can
be exploited further to prove that the Banach spaces $C(K)$ have no ucountable
biorthogonal systems (\cite{negrepontis}, \cite{hajek}).

The first question we considered was whether such phenomena can take place
 if the C*-algebra is made considerably noncommutative. One of our main results is:

\begin{theorem}\label{theorem-diamond} Assume $\diamondsuit$. There is a fully noncommutative
nonseparable scattered C*-algebra (of operators in $\B(\ell_2)$)  with no nonseparable commutative subalgebra and
with no uncountable irredundant set.
\end{theorem}
\begin{proof} Apply Theorems \ref{diamond-scheme} and \ref{theorem-main}.
\end{proof}

Here scattered C*-algebras are the noncommutative analogues of the scattered locally compact spaces. 
The condition of being fully noncommutative means that these algebras are 
``maximally noncommutative" among scattered algebras. These notions are reviewed in
Section 2.1. 

Another motivation for our project was the result of Todorcevic (\cite{stevo-irr}, \cite{stevo-biort})
that assuming Martin's axiom MA and the negation of the CH every uncountable Boolean algebra
has an uncountable irredundant set. Here the main question remains open:

\begin{question}\label{question-main} Is it consistent that every nonseparable
(AF, scattered) C*-algebra (of operators in $\B(\ell_2)$)
contains an uncountable irredundant set?
\end{question}

It should be added that even the commutative general case is open, since
the result of Todorcevic provides uncountable irredundant sets in $C(K)$s
only for $K$s totally disconnected and there can be nonmetrizable compact
spaces with no totally disconnected nonmetrizable compact subspace and similar
examples (see \cite{totally}).
 So it is natural
to restrict initially the attention in the noncommutative problem to
C*-algebras  corresponding to totally disconnected spaces, namely  to approximately finite 
dimensional C*-algebras (AF), i.e., where there is a dense subset which is the union of 
a directed family of finite dimensional C*-subalgebras (see \cite{farah-katsura}
for diverse notions of approximate finite-dimensionality in the nonseparable context).
Another natural narrowing of the question is to consider only the scattered C*-algebras since 
one of the conditions equivalent to being scattered for a C*-algebra of density $\omega_1$
is that each of its C*-subalgebras is AF.
Attempting to answer Question \ref{question-main} we obtained several results which 
shed some light on it. Let us discuss them below.

If $\A$ is AF C*-algebra of density equal to the first uncountable cardinal $\omega_1$, then
it can be written as $\A=\bigcup_{\alpha<\omega_1}\A_\xi$ where $\A_\xi\subseteq \A_{\xi'}$
for all $\xi<\xi'<\omega_1$ and each $\A_\xi$ is separable and AF. It follows from
the result of Thiel in \cite{thiel} (cf. \cite{olsen}, \cite{thiel-winter})
that each $\A_\xi$ is singly generated by one element
$A_\xi\in \A_\xi$. Hence in the set $\{A_\xi: \xi<\omega_1\}$ irredundant subsets are at most singletons.
So there is no chance to extract (possibly using some
additional forcing axioms) an uncountable irredundant set from an arbitrary  norm discrete
set  of cardinality $\omega_1$ of operators in $\B(\ell_2)$.

The AF hypothesis allows nevertheless to avoid sets of operators as above. Namely,
if $\A={\overline{\bigcup_{D\in \DD}\A_D}}$, where all $\A_D$s are finite-dimensional
and $\A_D\subseteq \A_{D'}$ whenever $D\leq D'$ for $D\in \DD$ and $(\DD, \leq)$ is directed,
then given any norm discrete $\{A_\xi: \xi<\omega_1\}\subseteq \bigcup_{D\in \DD}\A_D$,
which exists by the nonseparability of $\A$,
for every finite $F\subseteq \omega_1$ the set
$$X_F=\{\xi<\omega_1: A_\xi\in \A_F\}$$
is a finite superset of $F$, where $\A_F$
is the C*-subalgebra generated by $\{A_\eta: \eta\in F\}$. So, the search for
an uncountable irredundant set among $\{A_\xi: \xi<\omega_1\}$ is equivalent to
the search for an uncountable  $X\subseteq\omega_1$ such that $X_F\cap X=F$
 for every $F\subseteq X$. 
 
 However this combinatorial problem
 for a general  function from finite subsets of $\omega_1$ to themselves 
 has the negative solution\footnote{It is enough to take $X_F$ to be of the form 
 $Y\cap[(\max F)+1]$ where $Y\in\mu$ is of minimal rank which contains $F$
 and where $\mu$ is an $(\omega,\omega_1)$-cardinal as in \cite{2cardinals}.
 $\mu$ is originally due to Velleman (\cite{zfc-morass}). A  positive result
 for general functions is that given 
 $n\in \N$ and a function $\phi$ from finite subsets of the $n$-th
 uncountable cardinal  $\omega_n$ into countable subsets of $\omega_n$ there 
is an $n$-element set $X\subseteq \omega_n$ such that $\xi\not\in \phi(X\setminus\{\xi\})$
for any $\xi\in X$. In particular, this gives that any norm discrete subset
of  cardinality $\omega_n$ in any C*-algebra
has an irredundant subset of cardinality $n$.}. Nevertheless
passing to the second uncountable cardinal $\omega_2$ allows for a very general
consistency result:

\begin{theorem}\label{theorem-cohen} It is consistent that $2^\omega=\omega_2$
and for every norm discrete collection of operators 
 $(A_\xi: \xi<\omega_2)$ in $\B(\ell_2)$ 
  there is a subset $X\subseteq\omega_2$ 
 of cardinality $\omega_2$ such that $(A_\xi: \xi\in X)$ is irredundant.
\end{theorem}

This is not a mere consequence of $\B(\ell_2)$ having density $\omega_2$ because 
 by a result of Brech and Koszmider (\cite{biort-tams})
it is consistent that there  exists  a commutative C*-subalgebra of $\ell_\infty$
of density  $2^\omega=\omega_2$ with no uncountable irredundant set.
The cardinal $\omega_2$ in Theorem \ref{theorem-cohen} can be replaced by any regular cardinal bigger than $\omega_1$
but it is not known if the result of \cite{biort-tams} can be generalized to bigger cardinals 
 than  $\omega_2$. Combining \ref{theorem-cohen},
 \ref{theorem-diamond} and knowing that $\diamondsuit$ implies CH we obtain:

\begin{corollary} It is independent from ZFC whether there is
a norm discrete collection of operators (projections) $(A_\xi: \xi<2^\omega)$ 
in $\B(\ell_2)$ with no uncountable (of cardinality $2^\omega$) irredundant subcollection of size $2^\omega$. 

It is independent from ZFC whether there is
C*-subalgebra of
 $\B(\ell_2)$ of density $2^\omega$ with no uncountable (of size $2^\omega$) irredundant set.
\end{corollary}

The commutative results mentioned above, in fact, are most often of topological nature,
where the compact Hausdorff space under the consideration is the Stone space $K_\A$
of a Boolean algebra $\A$.
For example, the reason the above-mentioned  Boolean algebras have countable irredundance 
is that the spread\footnote{The spread of a topological space $K$, denoted by
$s(K)$ is the supremum over the cardinalities of discrete subspaces of $K$.} of 
$K_\A\times K_\A$ is countable as the finite powers of the mentioned
$K_\A$s are hereditarily separable. Namely, in general we have $irr(\A)\leq s(K_\A\times K_\A)$ 
which was first noted in \cite{heindorf} and easily follows from the characterization
of irredundant sets in the commutative case (Lemma \ref{char-comm-unital-irr}).
Also the Urysohn Lemma gives the inequality $s(K)\leq irr(C(K))$ for any
locally compact Hausdorff $K$. This argument cannot be transferred to the noncommutative setting
since we do not have so general noncommutative Urysohn Lemma (for noncommutative Uryshon Lemma see \cite{urysohn}).
  That is for 
constructing an irredundant set  of cardinality $\kappa$ in a C*-algebra $\A$ it is enough
to construct a sequence of states $(\tau_\alpha: \alpha<\kappa)$ and a sequence
of positive elements $(A_\alpha: \alpha<\kappa)$ of $\A$ such that $\tau_\alpha(A_\alpha)>0$
for all $\alpha<\kappa$ and $\tau_\alpha(A_\beta)=0$ for all distinct $\alpha, \beta<\kappa$
(Lemma \ref{lemma-biorthogonal}), but a weak$^*$ discrete set of pure states does not
produce the elements $A_\alpha$ as above due to the lack of the Urysohn Lemma for nonorthogonal
closed projections. In fact assuming the Proper Forcing Axiom, PFA every nonseparable scattered C*-algebra
has an uncountable weak$^*$ discrete set of pure states (Corollary \ref{cor-oca-discrete}), but
this does not help us in constructing an uncountable irredundant set and
answering Question \ref{question-main} in the positive in the scattered case.

A  bolder approach to Question \ref{question-main} would be to try to
answer the following question in the positive:

\begin{question}\label{problem-abelian}
Is it consistent  (with MA and the negation of CH) that every
 nonseparable scattered (or even AF) C*-algebra has
 a nonseparable commutative subalgebra in one of its quotients?
\end{question}

Note that the class of scattered C*-algebras is closed under quotients and subalgebras
and every locally compact scattered  Hausdorff space is totally disconnected, so
the positive answer to the question above and the MA result of Todorcevic mentioned above 
would give the positive answer to Question \ref{question-main} in the scattered case.

Known ZFC examples of nonseparable C*-algebras with no nonseparable commutative
subalgebras are the reduced group C*-algebra of an uncountable free group
as shown by Popa in \cite{popa} and the algebras of Akemann and Doner as shown in
\cite{tristan}. However the former is not AF (and has an
uncountable irredundant set corresponding to the free generators of the group) and the latter
has a nonseparable commutative quotient $c_0(\omega_1)$ (which also has
an obvious uncountable irredundant set). Perhaps the algebra of \cite{stable} could provide
the negative answer to Question \ref{problem-abelian}.

The reason our algebra from Theorem \ref{theorem-diamond}  does not contain a nonseparable
commutative C*-subalgebra is that given any discrete sequence of projections in certain dense
subalgebra there are two of them which have maximal commutator equal to $1/2$ (the fact
that $1/2$ is the maximal value is proved in \cite{stampfli}). However in such
an arbitrary sequence there are also two projections which 
almost commute (see Theorem \ref{theorem-main}), 
so in this sense our algebra is quite random, 
that is no pattern repeats on any uncountable norm discrete subset of elements. 
In fact such behaviour is already sensitive to infinitary combinatorics beyond 
ZFC determined by $\diamondsuit$ and Open Coloring Axiom
 (OCA)\footnote{For the statement of OCA see Definition \ref{oca}.},
namely we have:

\begin{theorem}\label{theorem-oca} Assume OCA.
For every $0<\varepsilon<1/2$ among any sequence of operators  $(A_\xi: \xi<\omega_1)$ in $\B(\ell_2)$
 there is an uncountable $X\subseteq \omega_1$ such that 
\begin{itemize}
\item for every distinct $\xi_1, \xi_2\in X$ we have  $[A_{\xi_1}, A_{\xi_2}]>1/2-\varepsilon$, or
\item for every $\xi_1, \xi_2\in X$ we have  $[A_{\xi_1}, A_{\xi_2}]<\varepsilon$.
\end{itemize}
However, assuming $\diamondsuit$  there is a scattered 
C*-algebra $\A\subseteq \B(\ell_2)$ (it is in particular AF) such that for every $0<\varepsilon<1/2$ among any
discrete  sequence of projections $(P_\xi: \xi<\omega_1)$ 
in $\A$ 
\begin{itemize}
\item there are $\xi_1<\xi_2<\omega_1$ such that $[P_{\xi_1}, P_{\xi_2}]>1/2-\varepsilon$,
\item there are $\xi_1<\xi_2<\omega_1$ such that $[P_{\xi_1}, P_{\xi_2}]<\varepsilon$.
\end{itemize}
\end{theorem}
\begin{proof} Apply Corollary \ref{cor-oca} and Theorem \ref{theorem-main}
\end{proof}

Another natural question related to uncountable irredundant sets
in general C*-algebras is the following:

 \begin{question}\label{question-mckenzie}$ $
 \begin{enumerate}
\item Is it true that $d(\A)\leq 2^{irr(\A)}$ holds for every C*-algebra
(every C*-algebra of type I )? 
\item Can there be arbitrarily big C*-algebras with no
uncountable irredundant sets?
\end{enumerate}
\end{question}

This is motivated by a Boolean result of McKenzie (see 4.2.3 of \cite{koppelberg}) which
says that a Boolean algebra has a dense subalgebra not bigger than its irredundance.
This result has been generalized by Hida in \cite{clayton} to all commutative C*-algebras which
implies that $irr(\A)\leq 2^{d(\A)}$  holds for  commutative $\A$.  We prove
 this inequality answering Question \ref{question-mckenzie} for scattered C*-algebras
in  our
 Theorem \ref{cardinal-inequality-scat}.  

In Section 2 we review scattered C*-algebras and constructions schemes 
which is an elegant framework to deal with some constructions using $\diamondsuit$ recently
 introduced by Todorcevic in \cite{stevo-scheme}.
It was already applied in several functional analytic, topological and combinatorial contexts
in \cite{stevo-scheme}, \cite{fulgencio-stevo}, \cite{fulgencio}. 
In Section 3 we prove basic facts concerning irredundant sets in commutative and noncommutative setting.
 In Section 4 we prove the OCA part of Theorem \ref{theorem-oca}.
Section 5 is devoted to defining and investigating the partial order of finite dimensional
 approximations to our algebra from Theorem \ref{theorem-diamond}. In the final
 Section 6 we use the appropriate construction schemes described in Section 2 to
 construct the algebra from Theorem \ref{theorem-diamond}.

Notation and the terminology of this paper should be standard, however, it draws from
diverse parts of mathematics like Boolean algebras, operator theory, set-theory, logic and
general topology. When in doubt one could  refer to textbooks like \cite{koppelberg},
\cite{murphy}, \cite{jech}, \cite{kunen}, \cite{engelking}. In particular
by an embedding (isomorphism onto its image) we mean $*$-monomorphism
($*$-isomorphism) of $C^*$-algebras which is not necessarily unital,
$\ell_2(X)$ denotes the Hilbert space of square summable complex functions defined on a set $X$,
$\B(\ell_2(X))$ denotes the C*-algebra of all bounded operators on $\ell_2(X)$,
$\ell_2=\ell_2(\N)$,
$\langle \cdot, \cdot \rangle$ denotes the scalar product, $\A_+$
denotes the set of positive elements of a C*-algebra $\A$, $1_\A$ denotes the unit of $\A$ and
$\widetilde \A$ the unitization of $\A$, $B_{\A^*}$ denotes the dual ball of the algebra
$\A$, 
 $[A, B]=AB-BA$ for
$A, B\in \B(\ell_2)$, 
$M_n$ denotes the C*-algebra
of $n\times n$ matrices for $n\in \N$, $C(K)$ denotes the C*-algebra of complex valued
continuous functions on a  compact $K$ and $C_0(X)$ the C*-algebra of complex valued
continuous functions vanishing at infinity on a  locally compact $X$, $\chi_U$ denotes 
the characteristic function of a set $U$, $Clop(K)$ denotes the family of clopen 
subsets of a space $K$, $\omega_n$ denotes the $n$-th uncountable cardinal for $n\in \N$,  
$[X]^n$ denotes the family of all $n$-element subsets of a set $X$, $[X]^{<\omega}$
denotes the family of all finite subsets of a set $X$;
 $X<Y$ means that $x<y$ for all $x\in X$ and $y\in Y$ where
$X, Y$ are sets  of ordinals.

We would like to thank Alessandro Vignati for his feedback on an earlier version of this paper.

\section{Preliminaries}

\subsection{Scattered C*-algebras}

The reason why scattered C*-algebras will play an important role in our investigation
of irredundant sets is that in such algebras irredundant sets 
can easily be replaced by irredundant sets of projections (Proposition \ref{irr-projections}),
in particular
the Boolean results pass to
the C*-algebraic ones (Corollary \ref{cor-boolean}). Moreover all
commutative results culminate around the scattered case which seems most basic.

Recall that a topological space is called scattered if it does not contain any perfect subset or
in other words if each (closed) nonempty subset has a relative isolated point.
The phenomena related to the scatteredness were already analysed by Cantor which resulted in
the notion of the Cantor-Bendixson derivative of a topological space (\cite{engelking}).
The Boolean algebra manifestation of these phenomena was discovered by Mostowski and Tarski in
\cite{mostowski-tarski} as what is today known as superatomic Boolean algebras.
 The importance of the class of Banach spaces of the form $C(K)$, 
where $K$ is scattered, already implicitly known in the 30ties, was first
systematically revealed in \cite{pelczynski-semadeni}. Its generalization, Asplund Banach spaces, started
to play an important role in Banach space theory since the 60ties.
It was Jensen in \cite{jensen} who first defined a scattered $C^*$-algebra but they
 were considered earlier by Tomiyama \cite{tomiyama} and Wojtaszczyk \cite{wojtaszczyk}. 
A recent survey \cite{cantor-bendixson} underlines the links of scattered C*-algebras with
 its Boolean algebraic and commutative predecessors.
Recall that a projection $p$ in a  C*-algebra is called minimal if and only if
$p\A p=\C p$, i.e., minimal projections generalize isolated points. The $*$-subalgebra of $\mathcal A$ 
generated by the minimal projections of $\mathcal A$
will be denoted $\I^{At}(\mathcal A)$. We have the following 
observation from \cite{cantor-bendixson}:

\begin{proposition}\label{atoms} Suppose that $\mathcal A$ is a $C^*$-algebra.
\begin{enumerate}
\item $\I^{At}(\mathcal A)$ is an ideal of $\mathcal A$,
\item $\I^{At}(\mathcal A)$ is isomorphic to a subalgebra of the algebra $\mathcal K(\mathcal H)$ of all 
 compact operators on a Hilbert space $\mathcal H$,
\item $\I^{At}(\mathcal A)$ contains all ideals of $\mathcal A$ which
are  isomorphic to a subalgebra of
$\mathcal K(\mathcal H)$ for some Hilbert space $\mathcal H$, 
\item if an ideal $\mathcal I\subseteq \A$ is essential and
isomorphic to a subalgebra of $\mathcal K(\mathcal H)$ for some Hilbert space $\mathcal H$,
then $\mathcal I=\I^{At}(\mathcal A)$.
\end{enumerate}
\end{proposition}

A selected list of conditions equivalent to being scattered 
and which are relevant to our paper is given below. Any of these conditions can be 
taken as the definition of a scattered algebra.

\begin{theorem}[\cite{jensen, jensen2, wojtaszczyk, lin, kusuda, tomiyama, cantor-bendixson}]\label{theorem1}
Suppose that $\mathcal A$ is a $C^*$-algebra. The following
conditions are equivalent:
\begin{enumerate}

\item\label{ec-homomorphic} Every non-zero $*$-homomorphic   image of $\mathcal A$ has a minimal projection.
\item\label{ec-sequence-new} There is an ordinal $ht(\mathcal A)$
and a continuous  increasing sequence of closed ideals $(\I_\alpha^{At}(\A))_{\alpha\leq ht(\mathcal A)}$
called the Cantor-Bendixson composition series for $\A$
such that  $\I_0=\{0\}$, $\I_{ht(\mathcal A)}= \mathcal A$ and
$$\I^{At}(\mathcal A /\I_\alpha^{At}(\A))=\{[a]_{\I_\alpha^{At}(\A)}: a\in \I_{\alpha+1}^{At}(\A)\},$$
for every $\alpha < ht(\mathcal A)$.

\item\label{ec-subalgebra} Every non-zero subalgebra of $\mathcal A$ has a minimal projection.
\item Every non-zero subalgebra has a projection, 
\item Every subalgebra of $\A$ has real rank zero,
\item\label{ec-no-interval} $\A$ does not contain a copy of the $C^*$-algebra $C_0((0,1])=\{f\in C((0,1]):
\lim_{x\rightarrow 0}f(x)=0\}$.
\item\label{ec-spectrum} The spectrum of every self-adjoint element is countable.

\end{enumerate}
\end{theorem}

\begin{definition}\cite{cantor-bendixson} A scattered C*-algebra is called thin-tall
if and only if $ht(\A)$ from Theorem \ref{theorem1} (2) is equal $\omega_1$ and
$\I_{\alpha+1}^{At}(\A)/\I_\alpha^{At}(\A)$ is separable for each $\alpha<\omega_1$.
\end{definition}

 In the nonseparable context we are especially interested in condition (2) which was introduced
 in \cite{cantor-bendixson} which gives an essential composition series corresponding to
 the Cantor-Bendixson derivative. 
A scattered C*-algebra is called fully noncommutative if and only if for all $\alpha<ht(\A)$
the algebra
$\I^{At}(\A/\I_\alpha)$ is *-isomorphic to  the algebra of
 all compact operators  on a Hilbert space. We have the following two observations from
 \cite{cantor-bendixson}:
 
 \begin{proposition}\label{chain} Suppose that $\A$ is
a scattered $C^*$-algebra. 
The following are equivalent:
\begin{enumerate}
\item $\A$ is fully noncommutative,
\item the ideals of $\A$ form a chain,
\item the centers of the  multiplier algebras of
any quotient of $\A$ are all trivial.
\end{enumerate}
\end{proposition}
 
 \begin{proposition}\label{atomic} Every  
 scattered $C^*$-algebra $\A$ is atomic, i.e., the ideal $\I^{At}(\A)$
 is essential.
\end{proposition}

 Recall that in a topological space a sequence of points
 $\{x_\xi: \xi<\kappa\}$ is called right-separated (left-separated) if and
 only if $x_\xi\not\in {\overline{\{x_\eta: \eta>\xi\}}}$ for all $\xi<\kappa$
 ($x_\xi\not\in {\overline{\{x_\eta: \eta<\xi\}}}$ for all $\xi<\kappa$).
 Left and right separated sequences play an important role in commutative set-theoretic
 topology because a regular space is hereditarily Lindel\"of (hereditarily separable)
 if it has no uncountable right-separated (left-separated) sequences.
 Additional axioms like $\diamondsuit$, CH, MA, PFA\footnote{For the statement
 of the Proper Forcing Axiom (PFA) or  Martin's Axiom (MA) we refer the reader, for example, to \cite{jech} or \cite{stevo-partition}.
 PFA implies among others MA, OCA and $2^\omega=\omega_2$.} have substantial impact on
 the existence of right or left separated sequences in regular topological spaces,
 for example PFA implies that there are no regular S-spaces, i.e., every regular
 topological space which has an uncountable right-separated sequence has
 an uncountable  left-separated sequence as well (Theorem 8.9 of \cite{stevo-partition}).

 \begin{proposition}\label{oca-discrete} Suppose that $\A$ is a thin-tall C*-algebra. Then the
 dual ball $B_{\A^*}$ of $\A^*$ contains an uncountable  right-separated sequence
 of pure states in the weak$^*$ topology. In particular under the Proper Forcing Axiom
 (PFA) the
 dual ball $B_{\A^*}$ of $\A^*$ contains an uncountable discrete set consisting of
 pure states.
 \end{proposition}
 \begin{proof} 
 Let $(\I_\alpha)_{\alpha<\omega_1}$ be the Cantor-Bendixson composition series
 of \ref{theorem1} (3). As $\I_{\alpha+1}/\I_\alpha$ is an essential
 ideal of $\A/\I_{\alpha}$ which is *-isomorphic with the 
 the algebra of all compact operators on $\ell_2$, we can embed
 $\A/\I_{\alpha}$ into $\B(\ell_2)$ with the range containing
 all compact operators. Take $\tau_\alpha$  to be a vector pure state on
 $\B(\ell_2)$ composed with the quotient map and the embedding. So $\tau_\alpha$
 is a pure state on $\A$ which is zero on $\I_\alpha$ and there is
 $A_\alpha\in \I_{\alpha+1}$ such that $\tau_\alpha(A_\alpha)=1$.
 Denote the set of all pure states on $\A$ by $P(\A)$.
 Now consider 
 $$U_\alpha=\{\tau\in P(\A): \tau(A_\alpha)>0\}.$$
 Note that if $\tau_\beta\in U_\alpha$ then $\beta\leq \alpha$, so
 $\{\tau_\alpha: \alpha<\omega_1\}$ is right-separated in the weak$^*$ topology.
 So $\{\tau_\alpha: \alpha<\omega_1\}$ contains and uncountable left-separated sequence
 by PFA (Theorem 8.9 or \cite{stevo-partition}). It is clear that a sequence which is
 both left and right separated is discrete.
\end{proof}

\subsection{Construction schemes}

In this section we   recall some definitions  and results from \cite{stevo-scheme}.

\begin{definition}\label{def-end-ext} Let $E, F\in [\omega_1]^{<\omega}$.
\begin{enumerate}
\item $F< E$ whenever $\alpha<\beta$ for all $\alpha\in F$ and $\beta\in E$,
\item $F\sqsubseteq E$  whenever there is $\alpha\in\omega_1$ such that $E\cap\alpha=F$ 
(we say that $F$ is an initial fragment of $E$ or that $E$ end-extends $F$),
\item $F\sqsubset E$ whenever $F\sqsubseteq E$ and $E\setminus F\not=\emptyset$.
\end{enumerate}
\end{definition}

\begin{definition}\label{def-systems} Let 
$\eta$ be an ordinal and let $(F_\xi: \xi<\eta)=\mathcal F\subseteq[\omega_1]^{<\omega}$.
\begin{enumerate}
\item $\F$ is cofinal if for all $E\in[\omega_1]^{<\omega}$ there is $F\in\mathcal F$
such that $E\subseteq F$,
\item $(F_\xi: \xi<\eta)$ is a $\Delta$-system of length $\eta$ with root $\Delta$ whenever
$F_\xi\cap F_{\xi'}=\Delta$ for all  $\xi<\xi'<\eta$,
\item  A $\Delta$-system $(F_\xi: \xi<\eta)$  with root $\Delta$ is increasing whenever
$F_\xi\setminus \Delta< F_{\xi'}\setminus \Delta$ for all  $\xi<\xi'<\eta$,
\item A subset of a $\Delta$-system is called a subsystem,
\item $\F|F=\{E\in \F: E\subsetneq F\}$ for $F\subseteq \omega_1$.
\end{enumerate}
\end{definition}

\begin{definition} A pair of sequences $(n_k)_{k\in\N}\subseteq \N$ and $(r_k)_{k\in\N}\subseteq \N$
is called allowed parameters if and only if
\begin{enumerate}
\item $r_0=r_1=n_0=0$
\item $n_k\geq 2$ for all
$k\in\N$. 
\item each natural value appears in the sequence $(r_k)_{k\in\N}\subseteq \N$ infinitely
many times
\item $r_{k+1}<m_k$ where $m_0=1$, $m_{k+1}=r_{k+1}+n_{k+1}(m_k-r_{k+1})$ for $k>0$.
\end{enumerate}
\end{definition}

\begin{definition}\label{def-scheme} 
A construction scheme with a pair of allowed 
parameters $(n_k)_{k\in\N}\subseteq \N$ and $(r_k)_{k\in\N}\subseteq \N$  
is a cofinal family $\F=\bigcup_{n\in \N}\F_n$ satisfying
\begin{enumerate}
\item $\F_0=[\omega_1]^1$,
\item If $k>0$ and $E, F\in \F_k$, then $|E|=|F|$ and $E\cap F \sqsubseteq E, F$
  and 
  $$\{\phi_{F, E}[G]: G\in \F|E\}=\F|F,$$
where $\phi_{F, E}: E\rightarrow F$ is the order preserving bijection between $E$ and $F$,
\item If $k\geq 0$ and $F \in \F_{k+1}$, then the maximal elements of $\F |F$  are in $\F_k$ and they form an increasing
$\Delta$-system of  length $n_{k+1}$ such that $F$ is its union. The family of all these maximal elements is called the canonical decomposition
of $F$.
\end{enumerate}
\end{definition}

\begin{definition}\label{def-captures}
Given a construction scheme $\F$, we say that an $F\in\F_{k}$ for $k>0$
 captures a $\Delta$-system $(s_i : i < n)$ of finite subsets of $\omega_1$ with root $s$ if the
canonical decomposition $(F_i: i<n_k)$ of $F$ with root $\Delta$ has the following properties:
\begin{enumerate}
\item  $n_k\geq n$, $s \subseteq  \Delta$, and 
$s_i \setminus  s \subseteq F_i \setminus \Delta$ for all $i < n$.
\item  $\phi_{F_i, F_j}[s_i] = s_j$ for all $i < j < n$.
\end{enumerate}
When $n = n_k$, we say that $F$ fully captures the $\Delta$-system.
\end{definition}

\begin{theorem}[\cite{stevo-scheme}] \label{diamond-scheme}Assume $\diamondsuit$. For any pair of 
allowed parameters $(n_k)_{k\in\N}$  and $(r_k)_{k\in\N}$ there is
a construction scheme $\F$  with these parameters and there is a partition $(P_n)_{n\in\N}$ 
of $\N$ into infinitely many infinite sets such that for every $n\in \N$ and every uncountable
$\Delta$-system $T$ of finite subsets of $\omega_1$  there exist arbitrarily large  $k\in P_n$
and $F\in\F_{k+1}$  
which fully captures a subsystem of $T$.
\end{theorem}

\section{Irredundant sets}

\subsection{Reducing irredundant sets to special ones}

Because of Weierstrass-Stone theorem for unital commutative C*-algebras
sometimes it is useful to consider a strengthening of being irredundant, where
the subalgebras we generate are unital. However this does not affect the
cardinalities of irredundant sets much:

\begin{lemma}\label{irr-unital}
Suppose that $\A$ is a unital C*-algebra and that $\mathcal X\subseteq \A$
is its nonempty irredundant set. Then there is $x_0\in \mathcal X$ such that no element
$x$ of $\mathcal X\setminus\{x_0\}$ belongs to the unital C*-subalgebra
generated by $\mathcal X\setminus\{x_0, x\}$. 
\end{lemma}
\begin{proof} If  no element
$x$ of $\mathcal X$ belongs to the unital C*-subalgebra
generated by $\mathcal X\setminus\{ x\}$ we are done by taking any element
of $\mathcal X$ as $x_0$. 

Otherwise let $x_0\in \mathcal X$ belong to  the unital C*-subalgebra
generated by $\mathcal X\setminus\{x_0\}$ so $x_0=\lambda1+y$ where
$y$ is in the subalgebra generated by $\mathcal X\setminus\{x_0\}$ and 
$\lambda\in \C\setminus\{0\}$.
Suppose that
there is $x \in \mathcal X\setminus\{x_0\}$ in the unital C*-subalgebra
generated by $\mathcal X\setminus\{x_0, x\}$, i.e., $x=\lambda'1+z$
where  $z$ is in the subalgebra generated by $\mathcal X\setminus\{x_0, x\}$ and 
$\lambda'\in \C\setminus\{0\}$. So $1$ is in the algebra generated by $\mathcal X\setminus\{x_0\}$,
but this shows that $x_0=\lambda1+y$ is in the subalgebra generated by $\mathcal X\setminus\{x_0\}$,
a contradiction with the fact that $\mathcal X$ is irredundant.
\end{proof}

Clearly any two orthogonal one-dimensional projections in $M_2$ form an irredundant set, however
each of them is in the unital C*-algebra generated by the other projection.

\begin{proposition}\label{irr-positive} Suppose that $\A$ is a  C*-algebra, $\kappa$
is an infinite cardinal  and $\{A_\xi: \xi<\kappa\}$ is an irredundant set in $\A$.
Then there is an irredundant set $\{B_\xi: \xi<\kappa\}$ consisting of positive elements of $\A$.
\end{proposition}
\begin{proof}
Given $X\subseteq \kappa$ 
let $\A_X$ be the C*-subalgebra of $\A$ generated by $\{A_\xi: \xi\in X\}$.
Clearly $(A_\eta+A_\eta^*)/2, (A_\eta-A_\eta^*)/2i\in \A_{\kappa\setminus\{\xi\}}$ for every 
$\eta\not=\xi$ and $(A_\xi+A_\xi^*)/2$ and $(A_\xi-A_\xi^*)/2i$
cannot both belong to $\A_{\kappa\setminus\{\xi\}}$. So $\{B_\xi: \xi<\kappa\}$ is irredundant set consisting of
self-adjoint elements, where
$B_\xi\in \{(A_\xi+A_\xi^*)/2, (A_\xi-A_\xi^*)/2i\}$ is such that 
$B_\xi\not\in \A_{\kappa\setminus\{\xi\}}$.

To prove that we can obtain the same cardinality
 irredundant set consisting of all positive elements,
by the above we may assume that the original $A_\xi$s are self-adjoint.
We have $A_\xi={A_\xi}_+-{A_\xi}_-$.   Note that 
there is $B_\xi\in \{{A_\xi}_+, {A_\xi}_-\}$ which does not 
belong to $\A_{\kappa\setminus \{\xi\}}$.
But ${A_\eta}_+= (|A_\eta|+A_\eta)/2$  and ${A_\eta}_-=(A_\eta-|A_\eta|)/2$ for $|A_\eta|=\sqrt{A_\eta^2}$
belong to $\A_\xi$ for all $\eta\in \kappa\setminus\{\xi\}$. So
$\{B_\xi: \xi<\kappa\}$ is irredundant, as required.
\end{proof}

The following proposition shows the role of being scattered while 
extracting irredundant sets consisting of projections.

\begin{proposition}\label{irr-projections} 
Suppose that $\A$ is a  scattered C*-algebra, $\kappa$
is an infinite cardinal  and $\{A_\xi: \xi<\kappa\}$ is an irredundant set in $\A$.
Then there is an irredundant set $\{P_\xi: \xi<\kappa\}$ consisting of projections.
\end{proposition}
\begin{proof}
 By Lemma \ref{irr-positive} we may assume that $A_\xi$s are self-adjoint.
 Let us adopt the notation $\A_X$ for $X\subseteq \kappa$ from the proof of Lemma \ref{irr-positive}.

Since subalgebras of scattered algebras are scattered, $\A_{\{\xi\}}$s are scattered for each
$\xi<\kappa$ and so of the form $C_0(K_{\{\xi\}})$ for some locally compact 
scattered $K_{\{\xi\}}$ which must be totally disconnected. It follows that
linear combinations of projections of $\A_{\{\xi\}}$s are norm
dense in $\A_{\{\xi\}}$s. Hence for each $\xi<\kappa$
there is a projection $P_\xi\in \A_{\{\xi\}}$ such that $P_\xi\not\in \A_{\kappa\setminus\{\xi\}}$.
It follows that $\{P_\xi: \xi<\kappa\}$ is irredundant.

\end{proof}

\subsection{Irredundant sets in commutative C*-algebras}

The following two lemmas characterize irredundant sets in commutative C*-algebras.

\begin{lemma}\label{char-comm-unital-irr} Suppose that $K$ is  compact Hausdorff space and
$\mathcal X\subseteq C(K)$ is such that no $f\in \mathcal X$ belongs to
the unital C*-subalgebra of $C(K)$ generated by $\mathcal X\setminus \{f\}$.
Then for each $f\in \mathcal X$ there are $x_f, y_f\in K$ such that $f(x_f)\not=f(y_f)$
but $g(x_f)=g(y_f)$ for any $g\in \mathcal X\setminus \{f\}$.

Consequently if $\mathcal X$ is a nonempty irredundant set in $C(K)$, then
there is $h\in \mathcal X$ such that $\mathcal X\setminus \{h\}$ has the above property.
\end{lemma}
\begin{proof} By the Gelfand representation we may assume that $C(K)$ is 
the unital C*-algebra generated by $\mathcal X$.
By the complex Stone-Weierstrass theorem the proper C*-subalgebra generated 
by $\mathcal X\setminus\{f\}$ does not separate a pair of points of $K$, say $x_f, y_f$.
But they must be separated by $f$ by the fact that $\mathcal X$ generated $C(K)$.

The last part of the lemma follows from Lemma \ref{irr-unital}
\end{proof}

\begin{lemma}\label{char-comm-irr} Suppose that $X$ is locally compact noncompact Hausdorff space and
$\mathcal X\subseteq C_0(X)$ is irredundant then for every $f\in \mathcal X$
there are $x_f, y_f\in X$ such that either
\begin{itemize}
\item $f(x_f)\not=0$ and $g(x_f)=0$ for all $g\in \mathcal X\setminus \{f\}$, or
\item  $f(x_f)\not=f(y_f)$
but $g(x_f)=g(y_f)$ for all $g\in \mathcal X\setminus \{f\}$.
\end{itemize}
Points $x_f$ satisfying the first case form a discrete subspace of $X$
\end{lemma}
\begin{proof} Let $K=X\cup\{\infty\}$ be the one-point compactification of $X$.
We will identify $C_0(K)$ with a C*-subalgebra of $C(K)$. Note that 
$\mathcal X$ satisfies the hypothesis of Lemma \ref{char-comm-unital-irr}, because
if $f=\lambda1+g$ for $f, g\in C_0(X)$, the unit would be in $C_0(X)$ which  contradicts
the hypothesis that $X$ is noncompact. So we obtain the pairs of points
$x_f, y_f\in K$ as in Lemma \ref{char-comm-unital-irr}. The first case of the lemma
corresponds to the situation when one of the points $x_f, y_f$ is $\infty$, say
$y_f$. But then $h(y_f)=0$ for all $h\in C_0(X)$ which implies
$f(x_f)\not=0$ and $g(x_f)=0$ for all $g\in \mathcal X\setminus \{f\}$.

Considering open sets $U_f=\{x\in X_0: f(x)\not=0\}$ we obtain neighbourhoods
witnessing the discretness of the set of $x_f$s satisfying the first case.
\end{proof}

In fact discrete subsets of $K$ provide strong irredundant subsets in $C(K)$:

\begin{remark}\label{remark-discrete} Suppose that $K$ is compact Hausdorff space and $D\subseteq K$
is discrete. For each $d\in D$ consider $f_d\in C(K)$ such that $f_d(d)\not=0$
and $f_d(d')=0$ for all $d'\in D\setminus\{d\}$. Then $f_d$ does not belong to
the ideal generated by $\{f_{d'}:d'\in D\setminus\{d\}\}$. In particular 
$\{f_{d}:d\in D\}$ is irredundant.
\end{remark}

One should, however, note that there could be dramatic gap between
the sizes of discrete subsets and the sizes of irredundant sets:

\begin{remark}\label{remark-split} Let $K$ be the split interval, i.e., $\{0, 1\}^\N\times \{0,1\}$
with the order topology induced by the lexicographical order. Then
$K$ has no uncountable discrete subset (in fact, $K$ is hereditarily separable and
 hereditarily Lindel\"of) but $C(K)$ has an irredundant set 
$\{\chi_{[{0^\N}^\frown0, x^\frown0]}: x\in \{0,1\}^\N\}$ of cardinality continuum.
\end{remark}

Most of the literature concerning implicitly or explicitly irredundant sets
are related to Boolean algebras. As shown in the following lemma the relationship
between Boolean irredundance and irredundance for C*-algebras is very close in the light of lemma
\ref{irr-unital}.

\begin{definition}\label{def-boolean-irr} A subset $\mathcal X$ of a Boolean algebra $\A$ is called Boolean irredundant
if for every $x\in \mathcal X$ the element $x$ does not belong to the Boolean subalgebra
generated by $\mathcal X\setminus\{x\}$.
\end{definition}

\begin{lemma}\label{lifting-boolean}
Suppose that $\A$ is a unital C*-algebra and $\B\subseteq \A$ is a Boolean algebra of projections
in $\A$ and
$\mathcal X\subseteq\B$ is Boolean irredundant. Then $\mathcal X$  is irredundant in $\A$.

Suppose that $K$ is a totally disconnected space and $\mathcal X\subseteq C(K)$ consists of projections
and no element of $x\in \mathcal X$ belongs to the unital C*-algebra generated by
$\mathcal X\setminus\{x\}$. Then $\mathcal X$ is Boolean irredundant in the Boolean algebra
$\{\chi_U: U\in Clop(K)\}$.
\end{lemma}
\begin{proof}
Let $\CC$ be the C*-subalgebra of $\A$ generated by $\B$. It is abelian, so it is of the form
$C(K)$ where K is the Stone space of $\B$. It is enough to prove that $\mathcal X$
 is irredundant in $\CC$.
But given a proper Boolean subalgebra, there are distinct ultrafilters on the superalgebra which
coincide on the subalgebra. These ultrafilters are the points of $K$ witnessing the irredundance of
$\F$ as in Lemma \ref{epi-irr}.
\end{proof}

For commutative scattered C*-algebras the relationship between Boolean and C*-algebraic
irredundance is even closer:

\begin{corollary}\label{cor-boolean} Suppose that 
$\A$ is an infinite superatomic Boolean algebra. Then the Boolean irredundance of $\A$
is the same as $irr(C(K_\A))$, where $K_\A$ is the Stone space of $\A$.
\end{corollary}
\begin{proof}
As $\A$ is infinite, its Boolean irredundance is infinite (just take an infinite pairwise disjoint collection). By the first part of Lemma \ref{lifting-boolean} the Boolean irredundance of $\A$
is not bigger than $irr(C(K_\A))$. On the other hand consider any infinite irredundant subset
$\mathcal X$ of $C(K_\A)$. $K_\A$ is scattered as $\A$ is superatomic and so $C(K_\A)$ is a scattered C*-algebra.
By Proposition \ref{irr-projections} there is an irredundant
subset $\mathcal Y$ of $C(K_\A)$ of the same cardinality as $\mathcal X$ and consisting
of projections in $C(K_\A)$. By removing at most one element of $\mathcal Y$, by the second part
of Lemma \ref{lifting-boolean} and Lemma \ref{irr-unital} it is a Boolean irredundant
set in the Boolean algebra $\{\chi_U: U\in Clop(K_\A)\}$ and this yields 
a Boolean irredundant set in $\A$.
\end{proof}

The above positively answers Question 3.10 (3) of \cite{dzamonja} in the case of a scattered space.

\begin{corollary}\label{cor-real-valued}
 Suppose that $K$ is an infinite Hausdorff compact space. Then
$irr^\#(C_\R(K))=irr(C(K))$ where $irr^\#(C_\R(K))$ is the supremum over the cardinalities
of sets $\mathcal X$ of real-valued continuous functions on $K$ such that no $f\in \mathcal X$
belongs to the real unital Banach algebra generated by $\mathcal X\setminus\{f\}$.
In particular the $\pi$-weight of $K$ is bounded by ${irr(C(K))}$ and the density of $C(K)$
is bounded by $2^{irr(C(K))}$.
\end{corollary}
\begin{proof} Let $X\subseteq C(K)$ be an infinite irredundant set. By Lemma \ref{irr-positive}
we may assume that it consists of real-valued (non-negative) functions. As in the proof of
Lemma \ref{irr-unital}, by removing at most one element we may assume that no $f\in \mathcal X$
belongs to the real unital C*-algebra generated by $\mathcal X\setminus\{f\}$. So
$irr^\#(C_\R(K))\geq irr(C(K))$.

Now given a set $\mathcal X$ as in the lemma, by the real unital Weierstrass-Stone theorem there are pairs of
points $x_f, y_f\in K$ such that $f(x_f)\not=f(y_f)$
but $g(x_f)=g(y_f)$ for any $g\in \mathcal X\setminus \{f\}$ (cf. \cite{invariants}).
hence $\mathcal X$ is an irredundant set in the C*-algebra $C(K)$ by Lemma 
\ref{char-comm-unital-irr}.

The last part of the corollary follows from Theorem 10 of \cite{clayton} where
$\pi(K)\leq irr^\#(C_\R(K))$ is proved and from the fact
that the weight of a regular space is bounded by the exponent of its $\pi$-weight 
(Theorem 3.3. of \cite{hodel}).
\end{proof}


\subsection{Irredundant sets in general C*-algebras}

Having developed 
the motivations in the previous section now we move to the irredundant sets in general, possibly noncommutative C*-algebras.

\begin{proposition}\label{orthogonal} Every infinite pairwise orthogonal collection of self-adjoint
elements in a C*-algebra is irredundant.
In particular,  every infinite dimensional C*-algebra contains an infinite irredundant set.
\end{proposition}
\begin{proof}
This follows form the fact that given a self-adjoint element $A$ of
a C*-algebra $\A$ the set $\{B\in \A: AB=BA=0\}$ is a C*-subalgebra of $\A$.
\end{proof}

\begin{proposition} Suppose that an infinite dimensional
 C*-algebra  $\A$ is a von Neumann algebra. Then $\A$ has 
 an irredundant set of cardinality continuum.
\end{proposition}
\begin{proof}
An infinite dimensional von Neumann algebra has an infinite pairwise orthogonal
collection of projections and so it contains the
commutative C*-algebra $\ell_\infty$ which is $*$-isomorphic to $C(\beta\N)$.
The Boolean algebra $\wp(\N)$ is isomorphic to $\{\chi_U: U\in Clop(\beta\N)\}$
and so the Boolean irredundance of $\wp(\N)$ is equal to the irredundance of $\ell_\infty$
by Lemma \ref{cor-boolean}.
By considering an almost disjoint family (or an independent family) of cardinality continuum
 of subsets of $\N$ we obtain an irredundant set of cardinality continuum.
\end{proof}

The following Proposition corresponds to Remark  \ref{remark-discrete}.

\begin{lemma}\label{lemma-biorthogonal} Suppose that $\A$ is a C*-algebra, $\kappa$ a cardinal 
 $\{A_\xi: \xi<\kappa\}\subseteq\A_+$ and $\{\tau_\alpha: \alpha<\kappa\}$ a family of
 states such that
 \begin{itemize}
 \item $\tau_\alpha(A_\alpha)>0$
 \item $\tau_\alpha(A_\xi)=0$ for $\xi\not=\alpha$.
 \end{itemize}
 Then $\{A_\xi: \xi<\kappa\}$ is irredundant.
\end{lemma}
\begin{proof}
As in the GNS construction one proves that $L_\alpha=\{A\in \A: \tau_\alpha(A^*A)=0\}$ 
is a left-ideal in $\A$, and in particular a C*-subalgebra. So $X_\xi\in L_\alpha$, where
$A_\xi=X_\xi^*X_\xi$ and so $A_\xi\in L_\alpha$ for all $\xi\not=\alpha$. However by Theorem 3.3.2. 
of \cite{murphy} we have
$$0<\tau_\alpha(A_\alpha)\leq\|\tau_\alpha\|\tau_\alpha(A_\alpha^*A_\alpha),$$
so $A_\alpha\not \in L_\alpha$.
\end{proof}

In the noncommutative case, for pure states $\{\tau_\alpha: \alpha<\kappa\}$
being discrete in the weak* topology 
does not yield in general the existence of positive elements $A_\alpha$
like in the lemma above, as the noncommutative Urysohn lemmas require extra hypotheses
(\cite{urysohn}). 

The following proposition is a version of the commutative characterizations
in Lemmas \ref{char-comm-unital-irr} and \ref{char-comm-irr}.
It could be interesting to remark that a version of the following proposition
where "representations" are replaced by "irreducible representations" implies the
noncommutative Stone-Weierstrass theorem, which remains a well-known open problem.
 One should note that
below one of the possibilities of the representation is to be constantly zero.

\begin{proposition}[\cite{epimorphisms}]\label{epi-irr} Suppose that $\A$ is a C*-algebra and 
$\mathcal X\subseteq \A$ is an irredundant set.
Then for all $a\in \mathcal X$ 
there are Hilbert spaces $H_a$ and  representations $\pi_a^1, \pi_a^2: \A\rightarrow \B(H_a)$
such that $\pi_a^1(a)\not=\pi_a^2(a)$ but
$$\mathcal X\setminus\{a\}\subseteq\{b\in A: \pi_a^1(b)=\pi_a^2(b)\}.$$
\end{proposition}

\subsection{Irredundance in scattered C*-algebras}

The following proposition shows that thin-tall algebras play a special role
 in the context of uncountable irredundant sets.
 
 \begin{proposition}\label{reduction-thin-tall} If there is a nonseparable
  scattered C*-algebra with no uncountable irredundant set, then it contains a thin-tall
  scattered C*-algebra.
 \end{proposition}
 \begin{proof} First note that by the characterization of
 subalgebras of the algebra of compact operators (\cite{invitation}) a C*-algebra which is isomorphic to
 a subalgebra of the algebra  of all 
 compact operators on a Hilbert space $\mathcal H$ but not isomorphic to
 a subalgebra of the algebra  of all 
 compact operators on the separable Hilbert space $\mathcal H$ must contain an uncountable
 pairwise orthogonal set which is irredundant by \ref{orthogonal}. So 
 if a scattered $\A$ has no uncountable irredundant set, then all the algebras
 $\I^{At}(\A/\I_\alpha)$ are *-isomorphic to a subalgebra of the algebra of
 all compact operators on the separable or finite dimensional Hilbert space, but
 as $\A$ is nonseparable $ht(\A)\geq\omega_1$ and so $\I_{\omega_1}$ is the 
 required thin-tall subalgebra of $\A$.  
 
 \end{proof}
 
 \begin{corollary}\label{cor-oca-discrete} Assume PFA. Suppose that $\A$ is a nonseparable scattered C*-algebra. Then
 there is an uncountable  weak$^*$ discrete set  of pure states of $\A$.
 \end{corollary}
 \begin{proof} First suppose that $\A$ has a quotient that contains an uncountable
 orthogonal set of projections. Then it is clear that we can find pure states which
 form  a weak$^*$ discrete set.
 Otherwise using the argument as in the proof of Proposition
 \ref{reduction-thin-tall} we may assume that $\A$ is thin-tall. 
 By Proposition \ref{oca-discrete} $\A$ has an uncountable weak$^*$ discrete set  of pure states.
 \end{proof}
 
Below we prove a simple noncommutative version of a  Theorem of 
 McKenzie (see 4.2.3 of \cite{koppelberg}):

 \begin{theorem}\label{cardinal-inequality-scat} If $\A$ is a scattered C*-algebra, then
 $$d(\A)\leq 2^{irr(\A)}.$$
 \end{theorem}
\begin{proof}
Let $\kappa$ be the minimal cardinal such that $\I^{At}(\A)$ is a 
subalgebra of the algebra of all compact operators on $\ell_2(\kappa)$.
By the characterization of
 subalgebras of the algebra of compact operators (\cite{invitation}) $\A$ must contain 
 pairwise orthogonal set of cardinality $\kappa$ which is irredundant by \ref{orthogonal}.
 So $\kappa\leq irr(\A)$. By the essentiality of $\I^{At}(\A)$ 
 which follows from Proposition \ref{atomic} we can embed $\A$ into $\B(\ell_2(\kappa))$,
 so $d(\A)\leq 2^\kappa\leq 2^{irr(\A)}$ as required.
 \end{proof}

\subsection{Extracting irredundant sets from a given collection of operators}

\begin{proposition}\label{prop-olsen} 
There is a collection of operators $(A_\xi: \xi<\omega_1)$ in $\B(\ell_2)$ 
which generates a nonseparable C*-subalgebra of $\B(\ell_2)$ with no two-element irredundant subset.
Any fully noncommutative thin-tall C*-algebra is generated by such a sequence.
 \end{proposition}
\begin{proof} 

Construct a fully noncommutative thin-tall C*-algebra
$\A$ as in Theorem 7.6. of \cite{cantor-bendixson}, in particular
with Cantor-Bendixson decomposition $(\I_\alpha^{At}(\A))_{\alpha<\omega_1}$ (see \ref{theorem1} (2)), where
$\I_{\alpha+1}^{At}(\A)$ is *-isomorphic to $\widetilde{\I_\alpha^{At}(\A)}\otimes \mathcal K(\ell_2)$.

By Theorem 8 of \cite{olsen} any C*-algebra of the form $\B\oplus\mathcal K(\ell_2)$
is singly generated if $\B$ is separable and unital. So for each $\alpha<\omega_1$ pick $A_\alpha$
to be a single generator of $\I_{\alpha+1}^{At}(\A)$.

An alternative approach which gives the final statement of the Proposition is
to use the fact that  scattered C*-algebras are locally finite dimensional (see \cite{farah-katsura}
for more on these notions in the nonseparable context)
in the sense that  each of its  finite  subsets can be approximated from
a finite dimensional C*-subalgebra (\cite{kusuda}, \cite{lin}). So $\I_\alpha^{At}(\A)$
is locally finite dimensional and separable for each $\alpha<\omega$ and so
AF. Thus the result of \cite{thiel} implies that $\I_\alpha^{At}(\A)$ is singly generated
for each $\alpha<\omega_1$. So pick $A_{\alpha+1}$ as before.
This completes the proof of the theorem.

\end{proof}

Using the free set lemmas like in \cite{enflo} one can prove
that given a discrete set of operators $(\A_\alpha)_{\alpha<\omega_n}$ for $n\in \N$
there is an $n$-element irredundant set. However there is much stronger
consistent extraction principle:

\begin{theorem} It is relatively consistent that whenever
$(A_\xi: \xi<2^\omega)$ is a collection of operators in $\B(\ell_2)$ which generates
a C*-algebra of density continuum, then there is a set $I\subseteq 2^\omega$ of cardinality
continuum such that $(A_\xi: \xi\in I)$ is irredundant.
\end{theorem}
\begin{proof}
To obtain the relative consistency we will use the method of forcing (see \cite{kunen}).
We start with the ground model $V$ satisfying the generalized continuum hypothesis (GCH) and
we will consider the generic extension $V[G]$ where $G$ is a generic set in the 
forcing $\PP=Fn(\omega_2, 2)$ for adding $\omega_2$ Cohen reals (see Chapter VIII \S2 of \cite{kunen}). 

Fix a ground model orthonormal basis $(e_n: n\in \N)$ for $\ell_2$ in $V$. In $V[G]$
let $(A_\xi: \xi<2^\omega)$
be as in the theorem.  By passing to a subset of cardinality $2^\omega=\omega_2$ 
and using the hypothesis that  $(A_\xi: \xi<2^\omega)$ is a collection of operators in $\B(\ell_2)$ which generates
a C*-algebra of density continuum, we may assume that $A_\xi$ does not belong to
the C*-algebra generated by the operators $(A_\eta: \eta<\xi)$ for each $\xi<\omega_1$.
Moreover by passing to a subsequence we may assume that there is a rational 
$\varepsilon>0$ such that
$\|A-A_\xi\|>\varepsilon$ for every $A$ in the C*-algebra generated by the operators $(A_\eta: \eta<\xi)$ for each $\xi<\omega_1$.

Each $A_\xi$ can be identified with an $\N\times \N$ complex valued matrix
$(\langle A_\xi(e_n), e_m\rangle)_{m, n\in \N}$. Let $\dot A_\xi$ be $\PP$-names in $V$
for these matrices. Using the standard argument of nice names, the countable
chain condition for $\PP$ and passing to a subsequence
using the $\Delta$-system lemma for countable sets which follows from the GCH, we may assume that
there are permutations $\sigma_{\xi, \eta}:\omega_2\rightarrow\omega_2$  which lift to
the automorphisms of $\PP$ and the permutations $\sigma_{\xi, \eta}'$ of $\PP$ names such that
$$\sigma_{\xi, \eta}'(\dot A_\eta)=\dot A_\xi,$$
and for every $\xi, \eta\in \omega_2$ we have that 
$$\PP\forces \phi(\dot x_1, ..., \dot x_k)\  \hbox{if and only if}\ \ 
\PP\forces \phi(\sigma_{\xi, \eta}'(\dot x_1), ..., \sigma_{\xi, \eta}'(\dot x_k))$$ for any formula
$\phi$ in $k\in \N$ free variables and any sequence
$\dot x_1, ..., \dot x_k$ of $\PP$-names for $k\in\N$ (7.13 \cite{kunen}). Using this for the formulas which say 
that the distance of $A_\xi$ from any element of
the C*-algebra generated by the operators $(A_\eta: \eta\in F)$ for any finite $F\subseteq \xi$
is bigger than $\varepsilon$,
we conclude  that $\PP$ forces that no $\dot A_\xi$  belongs the C*-algebra generated by 
any countable collection from $\{\dot A_\eta: \eta\not=\xi\}$ (by considering a permutation
of $\omega_2$ which moves $\xi$ above the countable set). This means that 
$\PP$ forces that no $\dot A_\xi$  belongs the C*-algebra generated by 
the remaining operators  $\{\dot A_\eta: \eta\not=\xi\}$, i.e., that the collection is irredundant as required.

\end{proof}

The above is a version of applying a standard  argument as in  \cite{stevo-irr} in the context of
Boolean irredundance.

\section{Commutators under OCA}

The main consistent construction of this paper presented in the following sections 
has a strong randomness properties. In this section we show that this randomness does not take place
for any uncountable collection of operators in $\B(\ell_2)$ under the assumption of
Open Coloring Axiom,  OCA.
We will follow the approach to the strong operator topology from the book \cite{davidson} 
of Davidson.
Thus we have:

\begin{definition}Let $H$ be a Hilbert space. The strong operator topology (SOT) on $\mathcal B(H)$ 
is defined as the weakest topology such that the sets 
$$S(a,x):=\{b\in B(H): ||(b-a)(x)||<1\}$$
are open for each $a\in \mathcal B(H)$ and $x \in H$ . We denote by $(\mathcal B(H),\tau_{sot})$
and $(\mathcal B(H)_1, \tau_{sot})$ respectively the space $\mathcal B(H)$
 and the unit ball of $\mathcal B(H)$ with the strong operator topology.
\end{definition}

\begin{proposition}\label{sep+metriz} If $H$ is a separable Hilbert space,
 then $(\mathcal B(H)_1, \tau_{sot})$ is metrizable and separable in the strong operator topology.
\end{proposition}
\begin{proof}
For metrizability see \cite{davidson}, Proposition I.6.3.
For the separability
 fix some orthonormal basis $(e_n)_{n\in \N}$ and consider finite
rank operators in the linear span of one dimensional operators of the form $v\otimes w$ where 
$v, w$ have finitely many nonzero rational coordinates with respect to $(e_n)_{n\in \N}$.
It is clear that such operators are SOT dense in $\mathcal B(l_2)_1$, as required.
\end{proof}

By the remarks on page 16 and 17 of \cite{davidson} we have the following:

\begin{lemma}\label{continuous}The multiplication on
$\mathcal B(H)_1$  is jointly continuous in the SOT topology and
so every polynomial\footnote{By a polynomial $P(x,y)$ 
we mean a expression in the form $P(x,y)=\sum_i a_i x^i + \sum_i b_i y^i +\sum_{i,j}c_{i,j}x^iy^j + 
\sum_{i,j}d_{i,j}y^ix^j +e_0$.}  is
SOT continuous on $\mathcal B(H)_1$.
\end{lemma}

We will follow the approach to the 
 Open Coloring Axiom (OCA) from \cite{aq},  page 55. Its weaker version was discovered by Abraham, 
Rubin and Shelah (\cite{abraham-oca}) and the final form 
 was introduced by Todorcevic in \cite{stevo-partition}. It is consistent with ZFC. 
In fact, it is a consequence of the Proper Forcing Axiom (PFA). 
See Theorem 8 of \cite{stevo-partition}.
Recall that 
$$[X]^2=\{\{x, y\}\subseteq X: x\not=y\}.$$
 It is well known that the original
form of OCA from \cite{stevo-partition} for subsets of the reals  is equivalent to the version for 
separable metric spaces as in \cite{aq}:

\begin{definition}[Todorcevic \cite{stevo-partition}]\label{oca} OCA denotes de following statement:
 If $X$ is a separable metric Hausdorff space and  $[X]^2 = K_0\cup K_1$ is a  partition
 with $K_0$ open\footnote{We call $K_0\subseteq [X]^2$ open if the symmetric set
$\{(x, y)\in X\times X: \{x, y\}\in K_0\}$
is open in $K\times K\setminus \Delta$ in the product topology, where $\Delta$ denotes the
 diagonal of $X\times X$.}, then either there is an uncountable $Y\subseteq X$ 
such that $[Y]^2\subseteq K_0$, or else $X=\bigcup_{n\in \N}X_n$ where $[X_n]^2\subseteq K_1$ for each $n\in \N$.
\end{definition}

\begin{theorem}[OCA]\label{maintheorem} Let $(A_\alpha)_{\alpha<\omega_1}$ 
be an uncountable family in $\mathcal{B}(l_2)$ and $P(x,y)$ be a polynomial satisfying
$\|P(A, B)\|=\|P(B, A)\|$ for all $A, B\in \mathcal{B}(l_2)$. Then given $\varepsilon>0$, 
either there is an uncountable $\Gamma_0\subset \omega_1$ such that $||P(A_\alpha, A_\beta)||\leq \varepsilon$ for 
every distinct $\alpha, \beta \in \Gamma_0$ or else there is an uncountable $\Gamma_1\subset \omega_1$
 such that $||P(A_\alpha, A_\beta)||>\varepsilon$ for every distinct  $\alpha, \beta \in \Gamma_1$.
\end{theorem}
\begin{proof}
As $X$ is uncountable, by passing
to an uncountable subset, we may assume that there is $M>0$ such that $\|A_\alpha\|\leq M$
for all $\alpha<\omega_1$.
 Let $X=\{A_\alpha: \alpha<\omega_1\}\subseteq M\B(\ell_2)_1$ and note that $M\B(\ell_2)_1$ is
 metric and separable by Proposition \ref{sep+metriz}.
Define 
$$K_0 = \{\{A,B\}\in [X]^2: ||P(A, B)||> \varepsilon\}$$
and $K_1=[X]^2\setminus K_0$. 

First note that the separability is hereditary for metric spaces, so $X$ is metric separable
as a subspace of $(M\mathcal B(l_2)_1, \tau_{sot})$.

Now note that $K_0$ is open. Indeed if $\|P(A, B)\|> \varepsilon$, then there
is $x\in \ell_2$ of norm one and $\delta>0$  such that $\|P(A, B)(x)\|>\varepsilon+\delta$. 
 Now if $P(A', B')\in S(P(A, B),x/\delta)$, we have $\|P(A', B')(x)-P(A, B)(x)\|<\delta$ and so
 $\|P(A', B')\|>\varepsilon$; Hence 
$\{\{A', B'\}\in [X]^2:  P(A', B')\in S(P(A, B),x/\delta)\}\subseteq K_0$.
But $(A, B)\in P^{-1}[S(P(A, B),x/\delta)]$ is open in $X\times X$ with
the product SOT topology by the continuity of $P$ 
(Lemma \ref{continuous}).

So we are in the position of applying the OCA. 
From  \ref{oca} we obtain the required  uncountable set $\Gamma_0$ or $\Gamma_1$.

\end{proof}

\begin{corollary}[OCA]\label{cor-oca}Let $(A_\alpha)_{\alpha<\omega_1}$ be an uncountable family in $B(l_2)$.
 Then given $\varepsilon>0$, either there is an uncountable 
$\Gamma_0\subset \omega_1$ such that $||[A_\alpha, A_\beta]||\leq \varepsilon$ for
 every $\alpha, \beta \in \Gamma_0$ or else there is an uncountable 
$\Gamma_1\subset \omega_1$ such that $||[A_\alpha, A_\beta]||>\varepsilon$ for every $\alpha, \beta \in \Gamma_1$.
\end{corollary}
\begin{proof}
Consider $P(x,y)=xy-yx$ and apply Theorem \ref{maintheorem}.
\end{proof}

\begin{remark}
Let us remark on two trivial versions of the above results. 
First  let $(A_n)_{n\in \N}$ be an infinite family in $B(l_2)$. 
Then given $\varepsilon>0$, either there is an infinite $\Gamma_0\subset \N$
 such that $||[A_n, A_m]||\leq \varepsilon$ for every $n, m \in \Gamma_0$
 or else there is an infinite $\Gamma_1\subset \N$ such that $||[A_n, A_m]||>\varepsilon$
 for every $n, m \in \Gamma_1$. This follows from the Ramsey theorem whose consistent generalization is the OCA.

Secondly note that if $(A_\alpha)_{\alpha<\omega_1}$ is an uncountable family
in a separable C*-subalgebra of $B(\ell_2)$, then by its second countability in the norm topology
 it follows that for every $\delta>0$ there is an uncountable $\Gamma_0\subseteq \omega_1$
 such that $\|A_\alpha-A_\beta\|<\delta$ for every $\alpha, \beta\in \Gamma_0$ and 
 so given any polynomial $P$ satisfying
 $P(x, x)=0$ and $\varepsilon>0$, by the norm continuity of $P$ there is an uncountable
  $\Gamma_0\subseteq \omega_1$
 such that $\|P(A_\alpha,P_\beta)\|<\varepsilon$ for every $\alpha, \beta\in \Gamma_0$.
 
 In fact, in the nontrivial cases of Theorem \ref{maintheorem} and Corollary \ref{cor-oca}
 when $(A_\alpha)_{\alpha<\omega_1}$ generates a nonseparable C*-subalgebra of
 $\B(\ell_2)$ we may assume that $(A_\alpha)_{\alpha<\omega_1}$ forms a norm discrete set.
\end{remark}

\section{The partial order of finite dimensional approximations}

\subsection{Notation}

The C*-algebras that we consider in the rest of this paper are 
 subalgebras of $\B(\ell_2( \omega_1\times \N))$. In fact, the subspaces $\ell_2(\{\xi\}\times \N)$
 of $\ell_2(\omega_1\times \N)$,
which we call columns will be
invariant for all our algebras, so our algebras could be  identified with subalgebras of 
$\Pi_{\xi<\omega_1}\B(\ell_2(\{\xi\}\times\N))$. Also the map
$\pi_\alpha: \Pi_{\alpha\leq \xi<\omega_1}\B(\ell_2(\{\xi\}\times\N))
\rightarrow \B(\ell_2(\{\alpha\}\times\N)$ applied to the appropriate quotients, will
be faithful (see Lemma \ref{quotientembedding} (3)). Thus the purpose of this presentation
of the algebras is related to the transparent structure of the Cantor-Bendixson composition series
(see Proposition \ref{thintallfully} (3)).

 For $X\subseteq  \omega_1\times \N$, we introduce the following
notation:
\begin{itemize}



\item $(e_{\xi, n}: \xi<\omega_1, n\in \N)$ is the canonical orthonormal 
basis of $\ell_2(\omega_1\times \N)$,
\item   the family of all operators $A$ in
$\B(\ell_2(\omega_1\times \N))$ such that
\begin{itemize}
\item  $\ell_2(X\cap (\{\xi\}\times \N))$ is $A$-invariant for all $\xi<\omega_1$,
\item  $A(e_{\xi, n})=0$ whenever $(\xi, n)\not \in X$,
\end{itemize}
will be denoted by $\B_{X}$,
\item  the unit of the C*-algebra $\B_X$ 
will be denoted by $P_X$, 
\item $1_{\xi, m, n}$ is the operator in $\B_{\omega_1\times \N}$ satisfying
$$
    1_{\xi, m, n}(e_{\eta, k}) =
    \begin{cases}
      e_{\xi, m} &  \text{if} \ k=n, \xi=\eta \\
      0        & \text{otherwise,}  \\
    \end{cases}
$$
\item if $A\in \B_{\omega_1\times\N}$  we define
$A|X=AP_X$, 
\item  if $A\in \B_{\omega_1\times\N}$  and $a\subseteq \omega_1$ we define
$A|a$   as $A|(a\times \N)$,
\item $\A|X=\{A|X: A\in \A\}$ for $\A\subseteq \B_{\omega_1\times\N}$ and $X\subseteq \omega_1\times \N$.
\end{itemize}

\subsection{The definition of the partial order of finite-dimensional approximations}

\begin{definition}\label{def-P}
We define a partial order $\PP$ consisting of
elements 
$$p=\big(a_p, \{n^p_\xi: \xi\in a_p\}, \{A^p_{\xi, m, n}: \xi\in a_p, n, m\in [0,n_\xi^p)\}\big),$$
where
\begin{enumerate}
\item $a_p$  is a finite subset of  $\omega_1$,
\item $n_\xi^p\in \N$  for each $\xi\in a^p$,
\item $A^p_{\xi, m, n}\in \B_{X_p}$ for each $\xi\in a_p$ and $n, m\in [0, n^p_\xi)$, where
$$X_p=\{(\xi, n): \xi \in a_p;\  n\in [0, n^p_\xi)\},$$
\item $A^p_{\xi, m, n}= (A^p_{\xi, m, n}| \xi)+1_{\xi, m, n}$ for each $\xi\in a_p$ and $n, m\in [0,n_\xi^p)$.
\end{enumerate}

The order $\leq_\PP=\leq$ on $\PP$  is defined by declaring $p\leq q$ if and only if:
\begin{enumerate}[(a)]
\item $a_p\supseteq a_q$,
\item $n_\xi^p\geq n_\xi^q$ for $\xi\in a_q$,

\item there is a (nonunital) *-embedding $i_{pq}: \B_{X^q}\to \B_{X^p}$ such
that $i_{pq}(A^q_{\xi, m, n})=A^p_{\xi, m,  n}$ for  all $\xi\in a_q$ and $m, n\in [0,n^q_\xi)$,
\item $i_{p, q}(A)|X_q=A$
for all $A\in \B_{X^q}$.
\end{enumerate}
\end{definition}

\begin{definition}\label{def-A-p} Suppose that $p\in \PP$ and $X\subseteq X_p$.
Then the $C^*$-subalgebra of $\B_{X_p}$ 
generated by $\{A^p_{\xi, m, n}: (\xi, m), (\xi, n)\in X\}$ is denoted by $\A^p_X$.
\end{definition}

\begin{lemma}\label{generating-B}For every 
$\alpha\in \omega_1$ and every $p\in \PP$ we have
$$\A_{X_p\cap (\alpha\times\N) }^p=\B_{X_p\cap (\alpha\times\N) }.$$
In particular $\A_{X_p}^p=\B_{X_p}$.

\end{lemma}
\begin{proof}
We will prove it by induction on $|a_p\cap\alpha|$. 
If $a_p\cap\alpha=\emptyset$, then both of the algebras are $\{0\}$.
Suppose $|a_p\cap\alpha|=n+1$ and we have proved the Lemma for 
every $q \in \PP$ and $\alpha<\omega_1$ such that $|a_q\cap\alpha|=n$. 
Let $\xi = max (a_p\cap\alpha)$. By the definition of $\B_{X_p}$
 we have that
 $\B_{X_p\cap (\alpha\times\N)}$ is *-isomorphic to $\B_{X_p \cap (\xi\times \mathbb N)}
  \oplus \B_{\{\xi\}\times \mathbb N}$. 
By the inductive hypothesis, $\B_{X_p \cap (\xi\times \mathbb N)}$ is 
generated by $\{A_{\eta,m,n}^p: \eta \in a_p\cap \xi; m,n\in [0,n^p_\xi)\}$. 
But by (4) in Definition \ref{def-P}, we have that 
$1_{\xi,m,n}=A^p_{\xi,m,n}-A$ for some $A\in \B_{X_p \cap (\xi\times \mathbb N)}$
and all $m,n\in [0, n_\xi^p)$.
 In particular, $\B_{\{\xi\}\times \mathbb [0, n_\xi^p)}$   is included in the algebra generated by 
$\{A_{\eta,m,n}^p: \eta \in a_p\cap\alpha; m,n\in [0,n^p_\eta)\}$.
 This together with the inductive hypothesis completes the proof.
\end{proof}

\begin{lemma}\label{norms-of-restrictions}
 Suppose that $\alpha<\omega_1$ and $p, q\in \PP$ satisfy $p\leq q$ and $A=i_{p, q}(B)$,
where $B\in \A^q_{X_q}$. Then 
$$\|A|[\alpha, \omega_1)\|=\|B|[\alpha, \omega_1)\|.$$
\end{lemma}
\begin{proof} Since $B|\alpha$ and $B|[\alpha, \omega_1)$ are in $A^q_{X_q}$ by Lemma \ref{generating-B},
we have $$\|A|[\alpha, \omega_1)\|=\|i_{p, q}(B)|[\alpha, \omega_1)\|=
\|i_{p, q}(B|\alpha)|[\alpha, \omega_1)+i_{p, q}(B|[\alpha, \omega_1))|[\alpha, \omega_1)\|.$$
But $B|\alpha\in A^q_{X_q\cap(\alpha\times\N)}$ by Lemma \ref{generating-B}, and this
generation
must be preserved by the isomorphism $i_{p, q}$, i.e., $i_{p, q}(B|\alpha)|[\alpha, \omega_1)=0$
and so 
$$\|A|[\alpha, \omega_1)\|=\|i_{p, q}(B|[\alpha, \omega_1)|[\alpha, \omega_1)\|\leq
\|i_{p, q}(B|[\alpha, \omega_1)\|.$$
Since $i_{p, q}$ is an embedding (in particular an isometry) , we conclude that
$$\|A|[\alpha, \omega_1)\|\leq\|B|[\alpha, \omega_1)\|.$$
The other inequality follows from Definition \ref{def-P} (c-d).
\end{proof}


\subsection{Density Lemmas}

In the terminology related to partial orders occurring in the theory of forcing
 a subset $\DD$ of a partial order $\QQ$ is said to be dense if for every $p\in \PP$
there is $d\in \DD$ satisfying $d\leq p$.  In what follows we usually need stronger
information for $\QQ=\PP$ namely that $a_d=a_p$.

\begin{lemma}\label{density1} Suppose that $\xi<\omega_1$.
Then 
$$\DD_\xi=\{p\in \PP: \xi\in a_p\}$$
is a dense subset of $\PP$.

\end{lemma}
\begin{proof}

 Let $q \in \PP$ be such that $\xi \notin a_q$. Define $p$ as follows:
\begin{itemize}
\item $a_p=a_q\cup \{\xi\}$,
\item $n_\eta^p = n_\eta^q$ for $\eta \in a_q$ and $n_\xi^p =1$,
\item $A_{\eta,m,n}^p= A_{\eta,m,n}^q$ for $\eta \in a_q$ and $A_{\xi,0,0}^p=1_{\xi,0,0}$.
\end{itemize}
It is clear  that $p\in \PP$.  Also $p\leq q$ as 
$Id_{\B_{X_q}}: \B_{X_q}\to \B_{X_p}$ is a *-embedding good for
$i_{p, q}$ in Definition \ref{def-P} (c).
\end{proof}

\begin{lemma}\label{density1.5}
Suppose that $\xi<\omega_1$;  $k\in \N$ and $q\in \PP$ is such that $\xi \in a_q$. Then there is
$$p\in \EE_{\xi, k}=\{p\in \PP: \xi\in a_p, n^p_\xi\geq k\}$$
such that $p\leq q$ and $a_p=a_q$.
\end{lemma}
\begin{proof}
\item Consider $q\in \PP$ such that $\xi \in a_q$ but $n_\xi^q<k$. Define $p$ as follows:
\begin{itemize}
\item $a_p=a_q$,
\item $n_\eta^p = n_\eta^q$ for $\eta \in a_p\setminus \{\xi\}$ and $n_\xi^p=k$,
\item $A_{\eta,m,n}^p=A_{\eta,m,n}^q$ $\eta \in a_q\setminus \{\xi\}$,
\item $A_{\xi,m,n}^p=A_{\xi,m,n}^q$ for $n,m \in [0, n_\xi^q)$,
\item $A_{\xi,m,n}^p = 1_{\xi,m,n}$ if $n,m\in [0, k)$ and $\{n, m\}\cap [n_\xi^q, k)\not=\emptyset$.
\end{itemize}
It is clear that  $p\in \PP\cap \EE_{\xi, k}$.
Also $p\leq q$ as 
$Id_{\B_{X_q}}: \B_{X_q}\to \B_{X_p}$ is a *-embedding good for
$i_{p, q}$ in Definition \ref{def-P} (c).

\end{proof}

\begin{lemma}\label{density2}
Suppose that  $q\in \PP$ and $X\subseteq X_q$  and that  $\alpha\in a_q$. 
Then there is $p\leq q$ such that 
$p\in \FF_{X, \alpha}$, where
$$\FF_{X, \alpha}=\{p\in\PP: \alpha\in a_p, \ X\subseteq X_p,  \ \hbox{and}\
  \ \forall A\in \A^p_X\ \   \|A|\{\alpha\}\|\geq \|A|[\alpha, \omega_1)\|\}.$$
Moreover,  $a_p=a_q$ and $n^p_\xi=n^q_\xi$
whenever $\xi\in a_p\setminus\{\alpha\}$.
\end{lemma}

\begin{proof}
Let $q \in \PP$. We may assume that $X=X_q$. If $\alpha=\max(a_p)$, then there is nothing to prove. 
So let  $a_q\setminus(\alpha+1)=\{\xi_1, ..., \xi_k\}$ for some $k\in \N$ and put
$$l=\sum\{n_{\xi_i}^q:\ 1\leq i\leq k\}.$$
Consider 
 $Y = X_q\cap ((\alpha,\omega_1)\times \N)$. Let $\phi: Y\to [ n_\alpha^q,  n_\alpha^q+l)$ be any bijection.
 We obtain
 a $*$-homomorphism $i:\B_{X_q}\to \B_{X_q\cup \big(\{\alpha\}\times[ n_\alpha^q,  n_\alpha^q+l)\big)}$ 
 given by
$i(A)=A+i_r(A)$ where $i_r:\B_{X_q}\to \B_{\{\alpha\}\times[ n_\alpha^q,  n_\alpha^q+l)}$ satisfies 
$$\langle i_r(A) (e_{\alpha, n_\alpha^q+\phi(\xi_i, k)}),
 e_{\alpha, n_\alpha^q+\phi(\xi_{i'}, k')}\rangle=
\langle A (e_{\xi_{i}, k}), e_{\xi_{i'}, k'}\rangle$$
for all $(\xi,k), (\xi', k')\in Y$ and every $A\in  \B_{X_q}$.  Define $p$ in the following way
\begin{itemize}
\item $a_p = a_q$,
\item $n_\xi^p = n_\xi^q$ if $\xi \in a_p\setminus \{\alpha\}$ and $n_\alpha^p = n_\alpha^q+l$,
\item $A_{\xi,m,n}^p = i(A_{\xi,m,n}^q)$ for $(\xi, m), (\xi, n) \in X_q$,
\item $A_{\alpha,m,n}^p = 1_{\alpha,m,n}$ if $\{m,n\}\cap [n_\alpha^q, n_\alpha^p)\neq \emptyset$.
\end{itemize}
It is clear from the construction
that  $p\in \PP$ as condition (4) of Definition \ref{def-P} is satisfied due to the fact that 
we change only $A^q_{\xi, m, n}$ for $\xi>\alpha$ on $\{\alpha\}\times\N$, 
and that (a), (b) of Definition \ref{def-P} are satisfied.

If we put $i_{p, q}=i$, condition (c) follows from the fact that $i$ is a $*$-embedding since 
$\{\alpha\}\times[ n_\alpha^q,  n_\alpha^q+n)\cap X_q=\emptyset$. We also have
  $i_{p, q}(A_{\xi,m,n}^q)=A_{\xi,m,n}^p$  for $(\xi, m), (\xi, n) \in X_q$. The construction yields
  (d) of Definition \ref{def-P}.

Finally to check the main assertion of the lemma note that 
by Lemma \ref{norms-of-restrictions} for any $A\in \B_{X_q}$ we have 
$$\|i_{p, q}(A)|\{\alpha\}\|= \max(\|i_{p, q}(A)|\{\alpha\}\|, 
\|i_{p, q}(A)|\{\alpha\}\times[ n_\alpha^q,  n_\alpha^q+n)\|)=$$
$$=\max(\|i_{p, q}(A)|\{\alpha\}\|, \|A|(\alpha, \omega_1)\|)=
\max(\|i_{p, q}(A)|\{\alpha\}\|, \|i_{p, q}(A)|(\alpha, \omega_1)\|)=$$
$$=\|i_{p, q}(A)|[\alpha, \omega_1)\|$$
for any  $A\in  \B_{X_q}$ as required since $X\subseteq X_q$.

\end{proof}

\begin{lemma}\label{open-dense}
Let $X\subseteq\omega_1\times\N$ be finite and $\alpha\in X$. If $q\in \FF_{X, \alpha}$ and
$p\leq q$, then $p\in \FF_{X, \alpha}$.
\end{lemma}
\begin{proof}
Let $A\in \A^p_{X}$. As $X\subseteq X_q$ we have that $A=i_{p, q}(B)$ for some
$B\in \A^q_X\subset \A^q_{X_q}$. First  note that by Lemma \ref{norms-of-restrictions}
$$\|A|[\alpha, \omega_1)\|=\|B|[\alpha, \omega_1)\|.$$
 Now
$\|B|[\alpha, \omega_1)\|\leq \|B|\{\alpha\}\|$ by the hypothesis that $q\in F_{X, \alpha}$.
But $\|B|\{\alpha\}\|\leq \|A|\{\alpha\}\|$ by the fact that $A|X^q=B$ by
Definition \ref{def-P} (d). So $\|A|[\alpha, \omega_1)\|\leq \|A|\{\alpha\}\|$ as required.

\end{proof}

\subsection{Basic amalgamations}

\begin{definition}\label{convenient} We say that two elements $p, q\in \PP$ are in the convenient position
(as witnessed by $\sigma:  a_p \to a_q$) if and only if 
$$\Delta:=a_p\cap a_q <a_p\setminus \Delta<a_q\setminus \Delta$$
and there is an order preserving bijection $\sigma:a_p \to a_q$ such that
\begin{itemize}
\item $n_\xi^p = n_{\sigma(\xi)}^q$ for $\xi\in a_p$,
\end{itemize}
and the $*$-isomorphism of $\B_{X_q}$ onto $\B_{X_p}$ induced by $\sigma$, denoted by
$j_\sigma$, which is given by
 $$\langle j_\sigma(A)(e_{\xi,k}), e_{\xi, l}\rangle =
 \langle A(e_{\sigma(\xi),k}), e_{\sigma(\xi), l}\rangle$$
for every  $(\xi,k), (\xi,l)\in X_p$ and $A\in \B_{X_q}$
satisfies  
\begin{itemize}
\item 
$j_\sigma(A_{\sigma(\xi),n,m}^q)=A_{\xi,n,m}^p$
for every $\xi \in a_p$, $n,m \in [0,n_{\xi}^p)$.
\end{itemize}

\end{definition}

\begin{lemma}\label{root-equality} Suppose that two elements $p, q\in \PP$ are in the convenient position
as witnessed by $\sigma:  a_p \to a_q$ and that  $\xi \in \Delta=a_p\cap a_q$. Then 
$A_{\xi,n,m}^q=A_{\xi,n,m}^p$ for every   $n,m \in [0,n_{\xi}^p)=[0,n_{\xi}^q)$.
\end{lemma}
\begin{proof}
Note that in Definition \ref{convenient} the bijection  $\sigma$ must be the identity on $\Delta$ 
because it is order-preserving and $\Delta$ is the initial fragment of both $a_p$ and $a_q$ and
so any $\xi\in \Delta$ must have the same position in both $a_p$ and $a_q$.
So $j_\sigma(A_{\xi,n,m}^q)=A_{\xi,n,m}^p$ and it is enough to prove that
 $j_\sigma(A_{\xi,n,m}^q)= A_{\xi,n,m}^q$.

For $\eta\in a_p\setminus \Delta$ we have 
$$\langle j_\sigma(A_{\xi,n,m}^q)(e_{\eta,k}), e_{\eta, l}\rangle =
 \langle A_{\xi,n,m}^q(e_{\sigma(\eta),k}), e_{\sigma(\eta), l}\rangle=0$$
for every $k, l\in \N$ such that  $(\eta,k), (\eta,l)\in X_p$ as $\sigma(\eta)\in a_q\setminus\Delta$.

On the other hand for $\eta\in\Delta$ we have $\sigma(\eta)=\eta$ and so
$$\langle j_\sigma(A_{\xi,n,m}^q)(e_{\eta,k}), e_{\eta, l}\rangle =
 \langle A_{\xi,n,m}^q(e_{\eta,k}), e_{\eta, l}\rangle$$
for every $k, l\in \N$ such that  $(\eta,k), (\eta,l)\in X_p$ as $\sigma(\eta)=\eta$
 by Definition \ref{convenient}.
Using Definition \ref{def-P} (4)
this proves the required $A_{\xi,n,m}^q=j_\sigma(A_{\xi,n,m}^q)=A_{\xi,n,m}^p$.
\end{proof}

\begin{lemma}\label{amal1} Suppose that $p, q\in \PP$  are in the convenient position
as witnessed by $\sigma: a_p\rightarrow a_q$. Then there is $r\leq p, q$ 
such that 
\begin{itemize}
\item $a_r = a_p\cup a_q$,
\item $n_\xi^r = n_\xi^p$ if $\xi \in a_p$ and $n_\xi^r = n_\xi^q$ if $\xi \in a_q$,
\item $i_{r, p}=Id_{\B_{X_p}}$,  $i_{r, q}=Id_{\B_{X_q}}.$
\end{itemize}
In particular,
\begin{itemize}
\item $A_{\xi, m, n}^r=A_{\xi, m, n}^p$ for each $\xi\in a_p$ and $n, m\in [0, n^r_\xi)$,
\item $A_{\xi, m, n}^r=A_{\xi, m, n}^q$ for each $\xi\in a_q$ and $n, m\in [0, n^r_\xi)$.
\end{itemize}
The element $r$ will be called the disjoint  amalgamation of $p$ and $q$.
\end{lemma}
\begin{proof}
Define $r$ as in the lemma. As $p,q \in \PP$, it is easy to see that $r\in \PP$.
To see that $r\leq p,q$ note that $Id_{\B_{X_p}}$ and   $Id_{\B_{X_q}}$
are $*$-embeddings into ${\B_{X_r}}$.
\end{proof}

\begin{lemma}\label{amal3}  Suppose that $p, q$  are two elements of 
$\PP$ in the convenient position as witnessed by $\sigma_{q,p}: a_p\to a_q$.
 Let $U\in \B_{X_{p}\cup X_q}$ 
be a partial isometry satisfying  $UU^*=U^*U=P_{X_p\setminus X_q}$,
where $P_{X_p\setminus X_q}$ is the projection on
the space spanned by $\{e_{\xi, k}: (\xi, k)\in X_p\setminus X_q\}$.
Then there is $r_U=r\leq p, q$ 
such that 
\noindent\begin{itemize}
\item $a_r = a_{p}\cup a_{q}$,
\item $n_\xi^r = n_\xi^{p}$ if $\xi \in a_{p}$, $n_\xi^r = n_\xi^{q}$ if $\xi \in a_{q}$,
\item $i_{r, p}=Id_{\B_{X_p}}$,
\item $i_{r, q}(A)=A+Uj_{\sigma_{q,p}}(A)U^*$ for all $A\in \B_{X_q}$,
\end{itemize}
in particular,
\begin{itemize}
\item $A_{\xi, m, n}^r=A_{\xi, m, n}^{p}$ for $\xi \in a_{p}$ and $m,n\in[0, n_\xi^{r})$,
\item $A_{\xi, m, n}^r=UA_{\sigma_{q,p}^{-1}(\xi), m, n}^{p}U^*
+A_{\xi, m, n}^{q}$ for $\xi \in a_{q}\setminus a_p$ and $m,n\in[0, n_\xi^{r})$.
\end{itemize}
The element $r_U$ will be called the $U$-including amalgamation of $p$ and $q$; if 
$U=P_{X_p\setminus X_q}$, then $r_U$ is called the including amalgamation.
\end{lemma}
\begin{proof} 
Define $r_U$ as in the lemma. It is clear by Definition \ref{def-P} applied to $p$ and $q$ that $r\in \PP$.
$r_U\leq p$ because $Id_{\B_{X_p}}: {\B_{X_p}}\to{\B_{X_r}}$ is a $*$-embedding. 
For $r_U\leq q$ we note that $A_{\xi, m, n}^r|X_q=A_{\xi, m, n}^q$ as 
$\big(UA_{\sigma_{q,p}^{-1}(\xi), m, n}^{p}U^*\big)|X_q=0$ since $UU^*=U^*U=P_{X_p\setminus X_q}$
and that  the formula $i_{r, q}(A)=A+Uj_\sigma(A)U^*$ for all $A\in \B_{X_q}$ defines
a *-embedding from ${\B_{X_p}}$ to ${\B_{X_r}}$. This is follows from the fact that
sending $A$ to $Uj_\sigma(A)U^*$ is a *-homomorphism since $\B_{X_p \setminus X_q}$
is $A^p_{X_p}$-invariant, so $i_{r, q}$ is a *-homomorphism.
But its kernel is null since $Uj_\sigma(A)U^*=(Uj_\sigma(A)U^*)|(X_p\setminus X_q)$ for all $A\in \B_{X_q}$.
\end{proof}

\begin{lemma}\label{one-half} Suppose that
$v_1, v_2$ are two orthogonal unit vectors of $\C^n$
for $n>1$.
 Then there is
a unitary $U\in M_n$ such that 
$$\|[UAU^*, A]\|=1/2$$
for every nonexpanding linear  $A\in M_n$ satisfying $A(v_1)=v_1$ and $A(v_2)=0$.
\end{lemma}
\begin{proof} Choose an orthonormal  basis $v_1, \dots v_n$ of $\C^n$  starting with $v_1, v_2$
and consider the orthogonal projection $P\in M_n$ onto the line containing $v_1$, so in particular we have
 $P(v_1)=v_1$ and $P(v_2)=0$. Let $U=V\oplus I_{n-2}$, $U^*=V^*\oplus I_{n-2}$, where 
$$V=V^*=\begin{pmatrix}
-1\over\sqrt2 & 1\over\sqrt2 \\
1\over\sqrt2 & 1\over\sqrt2
\end{pmatrix}. $$
So we obtain that 
$$UPU^*=\begin{pmatrix}
-1\over\sqrt2 & 1\over\sqrt2 \\
1\over\sqrt2 & 1\over\sqrt2
\end{pmatrix}
\begin{pmatrix}
1 & 0 \\
0 & 0
\end{pmatrix}
\begin{pmatrix}
-1\over\sqrt2 & 1\over\sqrt2 \\
1\over\sqrt2 & 1\over\sqrt2
\end{pmatrix}\oplus 0_{n-2}=
\begin{pmatrix}
1\over2 & -{1\over2} \\
-{1\over2} & 1\over2
\end{pmatrix}\oplus 0_{n-2}.
$$
Hence

$$[UPU^*, P]=UPU^*P-PUPU^*=$$
$$=\begin{pmatrix}
1\over2 &-{1\over2} \\
-{1\over2} & 1\over2
\end{pmatrix}\begin{pmatrix}
1 & 0 \\
0 & 0
\end{pmatrix}\oplus 0_{n-2}-
\begin{pmatrix}
1 & 0 \\
0 & 0
\end{pmatrix}
\begin{pmatrix}
1\over2 & -{1\over2} \\
-{1\over2} & 1\over2
\end{pmatrix}\oplus 0_{n-2}=
$$
$$
=\begin{pmatrix}
1\over2 & 0 \\
-{1\over2} & 0
\end{pmatrix}\oplus 0_{n-2}
-
\begin{pmatrix}
1\over2 & -{1\over2} \\
0 & 0
\end{pmatrix}\oplus 0_{n-2}=
\begin{pmatrix}
0 & {1\over2} \\
-{1\over2} & 0
\end{pmatrix}\oplus 0.
$$
And so $\|[UPU^*, P]\|=1/2$ and in particular 
\begin{itemize}
\item $[UPU^*, P](v_1)=(1/2)v_2$ and 
\item $[UPU^*, P](v_2)=(-1/2)v_1$.
\end{itemize}
Since $P$ equals $A$ on the space spanned by  $v_1$ and $v_2$,
and $U, U^*$ leave this space invariant, we have the same equalities for
$A$ instead of $P$, hence $\|[UAU^*, A]\|\geq 1/2$. The other inequality
follows from the fact that 
 $\|[B, C]\|\leq 1/2$ for any two  $B, C$ 
satisfying $0\leq B, C\leq 1$ by a result of Stampfli
(Corollary 2 of  \cite{stampfli}).
\end{proof}

\begin{lemma}\label{amal4}  Suppose that $p, q$  are two elements of 
$\PP$ in the convenient position as witnessed by $\sigma: a_p\to a_q$ such that 
$\Delta<a_{p}\setminus \Delta< a_{q}\setminus \Delta$. Suppose that $n_\xi^q= n\geq1$
for every $\xi\in a_q\setminus a_p$ and that 
$v_1=(v^1_0, ...v^1_{n-1}), v_2=(v^2_0, ...v^2_{n-1})$ are two orthogonal unit vectors 
of $\C^n$. Then  there is $r\leq p, q$ 
such that 
\noindent\begin{itemize}
\item $a_r = a_{p}\cup a_{q}$,
\item $n_\xi^r = n_\xi^{p}$ if $\xi \in a_{p}$, $n_\xi^r = n_\xi^{q}$ if $\xi \in a_{q}$,
\end{itemize}
and 
\begin{itemize}
\item $\|[i_{rq}(A), i_{rp}(j_\sigma(A))]\|=1/2$
\end{itemize}
for every nonexpanding $A\in \B_{X_q}$ such that there is $\xi\in a_q\setminus a_p$
with $A(\sum_{k<n}v^1_ke_{\xi,k})=\sum_{k<n}v^1_ke_{\xi,k}$ and 
$A(\sum_{k<n}v^2_ke_{\xi,k})=0$.
We call $r$ the $(v_1, v_2)$-anticommuting amalgamation of $p$ and $q$.
\end{lemma}
\begin{proof}
By Lemma \ref{one-half} for each $\xi\in a_q\setminus a_p$ for $\eta_\xi=\sigma^{-1}(\xi)$ there is
a unitary $U_\xi\in \B_{\{\eta_\xi\}\times[0,n)}$ such that 
$$\|[U_\xi(j_\sigma(A)|\{\eta_\xi\})U^*_\xi, j_\sigma(A)|\{\eta_\xi\}]\|
=1/2\leqno (*)$$
whenever $A\in \B_{X_q}$ is nonexpanding such that 
$$A(\sum_{k<n}v^1_ke_{\xi,k})=\sum_{k<n}v^1_ke_{\xi,k}\  \hbox{\rm and}
\ A(\sum_{k<n}v^2_ke_{\xi,k})=0.$$

Let $U\in \B_{X^p\cup X^q}$ be a partial isometry  such  that 
$U|(\{\eta_\xi\}\times [0,n))=U_\xi$ and  
$U$ is zero on the columns not in in $X^p\setminus X^q$, and $UU^*=P_{X^p\setminus X^q}$.
Consider the $U$-including amalgamation $r_U\leq p, q$  as in Lemma \ref{amal3}.

We claim that $r=r_U$ satisfies the lemma we are proving. Let
$A\in \B_{X_q}$ be nonexpanding and  $\xi\in a_q\setminus a_p$ be such that
 $A(\sum_{k<n}v^1_ke_{\xi,k})=\sum_{k<n}v^1_ke_{\xi,k}$ and 
$A(\sum_{k<n}v^2_ke_{\xi,k})=0$.
 Since $\ell_2({\{\xi\}\times[0,n_\xi^q)})$ for $\xi\in a_q\setminus a_p$ are invariant 
for $\B_{X^q}$  the operator $A|\{\xi\}$ is nonexpanding as well and so is
$j_\sigma(A)|\{\eta_\xi\}$.
By Lemma \ref{amal3} we have $i_{r, q}(A)=A+Uj_\sigma(A)U^*$
and $i_{r, p}(j_\sigma(A))=j_\sigma(A)$, so for $\eta_\xi=\sigma^{-1}(\xi)$ we have
$$[i_{r, q}(A), i_{r, p}(j_\sigma(A))]|(\{\eta_\xi\}\times[0,n))=
[U_\xi(j_\sigma(A)|\{\eta_\xi\})U_\xi^*, j_\sigma(A)|\{\eta_\xi\}],$$
So by (*) we have $\|[i_{r, q}(A), i_{r, p}(j_\sigma(A))]\|\geq1/2$. 
The other inequality follows from the maximality of $1/2$ (Corollary 2 of  \cite{stampfli}).

\end{proof}

\subsection{Types of $3$-amalgamations}

\begin{lemma}\label{amaltype1}  Suppose that $p_1, p_2, p_3$  are distinct 
 elements in $\PP$ which are pairwise in the convenient position as  witnessed by 
$\sigma_{j, i}: a_{p_i}\to a_{p_j}$ for $1\leq i<j\leq 3$ 
such that $\Delta<a_{p_1}\setminus \Delta< a_{p_2}\setminus \Delta <a_{p_3}\setminus \Delta$.
  Then there is $r\leq p_1, p_2, p_3$ 
satisfying
\begin{itemize}
\item $a_r = a_{p_1}\cup a_{p_2}\cup a_{p_3}$;
\item there is $n\in \N$ such that for each $\xi\in a_r$ we have
$$n=n_\xi^r >n'=\max\{ n_\xi^{p_i}: \xi \in a_{p_i},  1\leq i\leq 3\},$$
\item $$r\in \bigcap\{\FF_{X, \alpha}: X\in\{X_{p_1}, X_{p_2}, X_{p_3}\}, \alpha\in X\},$$
\end{itemize}
The element $r$ is called the amalgamation of $p_1, p_2, p_3$ of type 1.
\end{lemma}
\begin{proof}
Let $a_{p_1}=\{\alpha_1, ..., \alpha_k\}$ in the increasing order. 
Using Lemma \ref{density2} find $p_1\geq p^1_1\geq ... \geq p^k_1$ such that
$a_{p_1}=a_{p^k_1}$ and $a_{p^1_j}\in \FF_{X_{p_1}, \alpha_j}$ for $1\leq j\leq k$.
Now using Lemma  \ref{density1.5} several times find $q_1\leq p^k_1$ such that
$a_{q_1}=a_{p_1}$ and $n_\xi^{q_1}=n>n'$ for every $\xi\in a_{q_1}$. 

Now find $q_2, q_3\in \PP$ such that $q_2\leq p_2$ and $q_3\leq p_3$ 
and "isomorphic"  with $q_1$ i.e., with $a_{q_2}=a_{p_2}$, $a_{q_3}=a_{p_3}$
and where $q_1, q_2, q_3$ are pairwise in the convenient position as  witnessed by 
$\sigma_{j, i}: a_{q_i}\to a_{q_j}$ for $1\leq i<j\leq 3$. Note that by Lemma \ref{open-dense}
we have $$q_i\in \bigcap\{\FF_{X_{p_i}, \alpha}: 
\alpha\in X_{p_i}\}.$$

Now let $s_1\leq q_1, q_2$ and $s_2\leq q_1, q_3$ be the disjoint amalgamations as in
Lemma \ref{amal1}. Note that $s_1$ and $s_2$ are in the convenient position as witnessed
by $Id_{a_{p_1}}\cup\sigma_{3, 2}: a_{p_1}\cup a_{p_2}\rightarrow  a_{p_1}\cup a_{p_3}$
where $a_{s_1}\cap a_{s_2}=a_{p_1}$. So now let $r\leq s_1, s_2$ be the disjoint
amalgamation of $s_1$ and $s_2$ as in Lemma \ref{amal1}. Note that 
we have the final statement of the lemma by Lemma \ref{open-dense}.

\end{proof}

\begin{lemma}\label{amaltype2}  Suppose that $p_1, p_2, p_3$  are distinct 
 elements in $\PP$ which are pairwise in the convenient position as  witnessed by 
$\sigma_{j, i}: a_{p_i}\to a_{p_j}$ for $1\leq i<j\leq 3$ 
such that $\Delta<a_{p_1}\setminus \Delta< a_{p_2}\setminus \Delta <a_{p_3}\setminus \Delta$.
  Then there is $r\leq p_1, p_2, p_3$ 
satisfying
\begin{itemize}
\item $a_r = a_{p_1}\cup a_{p_2}\cup a_{p_3}$;
\item $n_\xi^r = n_\xi^{p_i}$ if $\xi \in a_{p_i}$ 
for $1\leq i\leq 3$,
\item $i_{r, p_1}=Id_{\B_{X_{p_1}}}$ 
\item $i_{r, p_2}(A)=A+j_{\sigma_{2, 1}}(A)|(X_{p_1}\setminus X_{p_2})$
 for all $A\in \B_{X_{p_2}}$,
\item $i_{r, p_3}(A)=A+j_{\sigma_{3, 1}}(A)|(X_{p_1}\setminus X_{p_3})$
 for all $A\in \B_{X_{p_3}}$,
\end{itemize}
In particular 
$$i_{r, p_3}(A)i_{r, p_2}(j_{\sigma_{3, 2}}(A))=i_{r, p_1}(j_{\sigma_{3, 1}}(A))^2$$
for every $A\in \B_{X_{p_3}}.$
The element $r$ is called the amalgamation of $p_1, p_2, p_3$ of type 2.
\end{lemma}
\begin{proof} 
First consider $s_2\leq p_1, p_2$ and $s_3\leq p_1, p_3$ which are
the including amalgamations of $p_1, p_2$ and $p_1, p_3$ as in Lemma \ref{amal3}.
It is clear that $s_1$ and $s_2$ are in the convenient position as witnessed by
$Id_{a_{p_1}}\cup \sigma_{3, 2}: a_{p_1}\cup a_{p_2}\rightarrow a_{p_1}\cup a_{p_3}$
Now let $r$ be the disjoint amalgamation of $s_1$ and $s_2$ as in Lemma \ref{amal1}.
The properties of $r$ follow from Lemma \ref{amal3} and Definition \ref{def-P}.

To prove the last statement of the lemma note that
$i_{r, p_3}(A)i_{r, p_2}(j_{\sigma_{3, 2}}(A))=\big(A+j_{\sigma_{3, 1}}(A)|(X_{p_1}\setminus X_{p_3})\big)
\big(j_{\sigma_{3, 2}}(A)+j_{\sigma_{3, 1}}(A)|(X_{p_1}\setminus X_{p_3})\big)=(j_{\sigma_{3, 1}}(A))^2=
i_{r, p_1}(j_{\sigma_{3, 1}}(A))^2$.
\end{proof}

\begin{lemma}\label{amaltype3}  Suppose that $p_1, p_2, p_3$  are distinct 
 elements in $\PP$ which are pairwise in the convenient position as  witnessed by 
$\sigma_{j, i}: a_{p_i}\to a_{p_j}$ for $1\leq i<j\leq 3$ 
such that $\Delta<a_{p_1}\setminus \Delta< a_{p_2}\setminus \Delta <a_{p_3}\setminus \Delta$
and $n_\xi^{p_i}=n$ for some $n>1$ and each $i\in \{1,2,3\}$ and that 
$v_1=(v^1_0, ...v^1_{n-1}), v_2=(v^2_0, ...v^2_{n-1})$ are two orthogonal unit vectors 
of $\C^n$.
  Then there is $r\leq p_1, p_2, p_3$ 
satisfying
\begin{itemize}
\item $a_r = a_{p_1}\cup a_{p_2}\cup a_{p_3}$;
\item $n_\xi^r = n_\xi^{p_i}=n$ if $\xi \in a_{p_i}$ 
for $1\leq i\leq 3$,
\item  $$\|[i_{r, p_m}(A), i_{r, p_1}(j_{\sigma_{m, 1}}(A))]\|=1/2,$$
for $m=2, 3$ and for every nonexpanding $A\in \B_{X_{p_m}}$ such that there is $\xi\in a_m\setminus a_{p_1}$
with $A(\sum_{k<n}v^1_ke_{\xi,k})=\sum_{k<n}v^1_ke_{\xi,k}$ and 
$A(\sum_{k<n}v^2_ke_{\xi,k})=0$.
\end{itemize}
The element $r$ is called the amalgamation of $p_1, p_2, p_3$ of type 3 for vectors $v_1$and $v_2$.
\end{lemma}
\begin{proof} First consider $s_2\leq p_1, p_2$ and $s_3\leq p_1, p_3$ which are
the $(v_1, v_2)$-anti-commuting amalgamations of $p_1, p_2$ and $p_1, p_3$ as in Lemma \ref{amal4}.
It is clear that $s_1$ and $s_2$ are in the convenient position as witnessed by
$Id_{a_{p_1}}\cup \sigma_{3, 2}: a_{p_1}\cup a_{p_2}\rightarrow a_{p_1}\cup a_{p_3}$
Now let $r$ be the disjoint amalgamation of $s_1$ and $s_2$ as in Lemma \ref{amal1}.
The properties of $s_1$ and $s_2$ from the $(v_1, v_2)$-anti-commuting amalgamations $s_1$ and $s_2$
pass to $r$ by Definition \ref{def-P} (d).
\end{proof}

\subsection{Inductive limits of directed families in $\PP$}
In this section we adopt the terminology where a directed set is a partial
order $(X, \leq)$ where for any two $x, y\in X$ there is $z\in X$ such that $z\leq x, y$.
In this section we will consider inductive limits $\A^\GG$ of systems $( \A^p_{X_p}: p\in \GG)$ where $\GG\subseteq \PP$
is a directed subset of $\PP$ with the order $\leq=\leq_\PP$. Here  for $p\leq q$ the embeddings
$i_{pq}: \A^q_{X_q}\rightarrow \A^p_{X_p}$ are given by Definition \ref{def-P} (c), i.e., they
satisfy $i_{pq}(A^q_{\xi, m, n})=A^p_{\xi, m, n}$ for $\xi \in a_q$ and $n, m\in [0, n_\xi^q)$.
Formally we define $\A^\GG$ differently in order to work with its convenient representation
in $\B_{\omega_1\times\N}$ but then, in Lemma \ref{inductive-limit}
 we prove that the constructed algebra is the corresponding inductive limit.
 
 \begin{definition} We say that $\GG\subseteq \PP$ is covering if and only if
 $$\omega_1\times\N\subseteq \bigcup\{X_p: p\in \GG\}.$$
 \end{definition}

\begin{definition}\label{def-A-xi-G} Suppose that $\GG\subseteq \PP$ is directed and covering. Then 
$A^\GG_{\xi, n, m}\in \B_{\omega_1\times\N}$ is given by
$$\langle A^\GG_{\xi, n, m}(e_{\eta, k}), e_{\eta, l}\rangle
=\langle A^p_{\xi, n, m}(e_{\eta, k}), e_{\eta, l}\rangle$$
for any (all) $p\in \GG$ such $(\eta, k), (\eta, l), (\xi, n), (\xi, m)\in X_p$.
\end{definition}

Note that $A^\GG_{\xi, n, m}$ are well-defined if $\GG$ is directed and covering. This is because
given two $p, p'\in \GG$ such that $(\eta, k), (\eta, l), (\xi, n), (\xi, m)\in X_p, X_{p'}$ 
there is $q\leq p, p'$ which implies that  $X_p, X_{p'}\subseteq X_q$ and so
$$\langle A^p_{\xi, n, m}(e_{\eta, k}), e_{\eta, l}\rangle
=\langle A^q_{\xi, n, m}(e_{\eta, k}), e_{\eta, l}\rangle=
\langle A^{p'}_{\xi, n, m}(e_{\eta, k}), e_{\eta, l}\rangle$$ by Definition \ref{def-P} (c-d). 
The following definition is parallel to Definition \ref{def-A-p}:

\begin{definition}\label{def-A-G} Suppose that $\GG\subseteq \PP$ is directed and covering. 
$\A^\GG$ is  the subalgebra of $\B_{\omega_1\times\N}$ 
 generated by the operators
$A_{\xi, m, n}^\GG$  for all $\xi\in \omega_1$ and $m, n\in \N$.

   Let $X$ be a subset of $\omega_1\times \mathbb{N}$. 
We define $\A_X^\GG$ to be the $C^*$-subalgebra of 
$\A^\GG$ generated by $(A_{\xi, m, n}^\GG: (\xi,n), (\xi, m)\in X)$.
In particular, for every $\alpha<\omega_1$, 
by $\A_\alpha^\GG$ we mean  the $C^*$-subalgebra of 
$\A^\GG$ generated by $\{A_{\xi,m,n}^\GG: \xi<\alpha, m,n \in \N\}$.

\end{definition}

\begin{lemma}\label{restriction1}
Suppose that $\GG\subset\PP$ is directed and covering and $p\in \GG$.
 There is a $*$-embedding $i_{\GG, p}: \A^p_{X_p}\to \A_{X_p}^\GG$
such that 
\begin{enumerate}
\item $i_{\GG, p}(A_{\xi, m, n}^p)=A_{\xi, m, n}^\GG$ and
\item $i_{\GG, p}(A_{\xi, m, n}^p)|X_p=A_{\xi, m, n}^p$
\end{enumerate} 
for every $\xi, n, m$ such that  $(\xi,n), (\xi, m)\in X$.
\end{lemma}
\begin{proof} By Definitions \ref{def-P} and \ref{def-A-xi-G},
a map sending $A_{\xi, m, n}^p$ to $A_{\xi, m, n}^\GG$ extends 
to a $*$-homomorphism of $\A^p_{X_p}$ into $\A^\GG_{X_p}$.
Its kernel must be  null as the kernels of $i_{q, p}$ for $q\leq p$ are null.

To prove the second part of the lemma, use the first part and
Definition \ref{def-A-xi-G}.
\end{proof}

\begin{lemma}\label{inductive-limit} Suppose that $\GG\subseteq \PP$ is directed and covering.
There is a $*$-isomorphism $j$ of $\A^\GG$ and the inductive limit $lim_{p\in \GG}\A^p_{X_p}$
of the system $(\A^p_{X_p}: p\in \GG)$ with maps $(i_{p, q}: p\leq q)$ such that
$$j(A^\GG_{\xi, n, m})=\lim_{p\in \GG}A^p_{\xi, n, m}$$
for each $\xi\in \omega_1$ and $m, n\in \N$. 
\end{lemma}
\begin{proof}
As in Ex. 1. Chapter 6 of \cite{murphy} it is enough to prove that 
for every $ p, q\in \GG$ satisfying $p\leq q$ the diagram
\[\begin{tikzcd}
\A^q_{X_q} \arrow{r}{i_{\GG, q}} \arrow[swap]{d}{i_{p, q}} & A_{X_q}^\GG \arrow{d}{\subseteq} \\
\A^p_{X_p}\arrow{r}{i_{\GG, p}} & A_{X_p}^\GG
\end{tikzcd}
\]
commutes. This follows from the fact that by Definition \ref{def-A-xi-G} we have
$i_{\GG, p}(i_{p, q}(A_{\xi, n, m}^q))=\A^\GG_{\xi, n, m}=i_{\GG, q}(A^q_{\xi, n, m})$
for $\xi, m, n$ such that $(\xi, m), (\xi, n)\in X_q$. But these elements
generate $\A^q_{X_q}$.
\end{proof}

\begin{definition}\label{def-F-rich} A family $\GG\subseteq \PP$ is 
called $\FF$-rich if and only if $\GG$ is directed, covering 
and $\FF_{X, \alpha}\cap\GG \not=\emptyset$ for every finite $X\subseteq\omega_1\times\N$ and
$\alpha\in X$, where $\FF_{X, \alpha}$s are defined in Lemma \ref{density2}.
\end{definition}

\begin{lemma}\label{quotientembedding}
Let $\GG\subseteq \PP$ be an $\FF$-rich family. Then for every $\alpha<\omega_1$
the following hold:
\begin{enumerate}
\item $\A_\alpha^\GG$ is an ideal of $\A^\GG$ equal to $\{A\in \A^\GG: A|[\alpha, \omega_1)=0\}$,
\item there is a *-isomorphism $j_\alpha: \A^\GG/\A_\alpha^\GG\to \A^\GG|[\alpha, \omega_1)$,
\item the representation $\pi_\alpha: \A^\GG|[\alpha, \omega_1)\to \A^\GG|\{\alpha\}$ 
given by $\pi_\alpha(A)=A|\{\alpha\}$ is faithful.
\end{enumerate}

\end{lemma}
\begin{proof}
As $\ell_2(\{\xi\}\times\N)$ are $\A^\GG$-invariant, it is clear that sending
$A\in \A^\GG$ to $A|[\alpha, \omega_1)$ is a $*$-homomorphism. So for (1) and (2) we are left
with proving that its kernel is equal to $\A^\GG_\alpha$.

First note that the kernel contains every
generator $A^\GG_{\xi, n, m}$ for $\xi<\alpha$ and $m, n\in \N$ of $\A_\alpha^\GG$
and so includes $\A_\alpha^\GG$. 
This is true by Definition \ref{def-P} (4).

For the other inclusion let $A\in \A^\GG$ satisfy $A|[\alpha, \omega_1)=0$. Since $\A^\GG$ is the inductive
limit of $\A^p_{X_p}$s for
$p\in \GG$ by Lemma \ref{inductive-limit}, for every $\varepsilon>0$ there is $p\in \GG$ and $B\in \A^p_{X_p}$ such that
$\|i_{\GG, p}(B)-A\|<\varepsilon$ and so $\|i_{\GG, p}(B)|[\alpha, \omega_1)\|<\varepsilon$. 
By Lemma \ref{generating-B}, $B|\alpha\in \A^p_{X_p\cap(\alpha\times\N)}\subseteq
\A^p_{X_p}$ and $B|[\alpha, \omega_1)\in \A^p_{X_p}$, so we can apply $i_{\GG, p}$ to them.
By Lemma \ref{norms-of-restrictions} and Definition \ref{def-A-xi-G} we have that
$$\|i_{\GG, p}(B|[\alpha,\omega_1))\|=\|i_{\GG, p}(B)|[\alpha,\omega_1)\|.$$
So we have
$$\|A-i_{\GG, p}(B|\alpha)\|=\|A-i_{\GG, p}(B)+i_{\GG,p}(B|[\alpha,\omega_1))\|\leq$$
$$\leq \|A-i_{\GG, p}(B)\|+\|i_{\GG, p}(B)|[\alpha,\omega_1)\|\leq 2\varepsilon.$$
But $i_{\GG, p}(B|\alpha)\in \A^\GG_\alpha$ since $B|\alpha\in \A^p_{X_p\cap(\alpha\times\N)}$.
 As $\varepsilon>0$ was arbitrary
and $i_{\GG, p}(B|\alpha)\in \A_\alpha^\GG$,
we conclude that $A\in \A_\alpha^\GG$, completing the proof of (1) and (2).

To prove (3) first note that since $\ell_2(\{\alpha\}\times\omega_1)$ is $\A^\GG$-invariant, 
it is clear that
$\pi_\alpha$ is a representation of $\A^\GG|[\alpha, \omega_1)$.  Now suppose that
$A\in\A^\GG_{X_q}$ for $q\in\GG$. By Lemma \ref{restriction1} there is $B\in \A^q_{X_q}$
such that $i_{\GG, q}(B)=A$.
 Since $\GG$ is assumed to be $\FF$-rich, by Lemmas \ref{density2} and \ref{open-dense}
 there is $p\in \FF_{X_q, \alpha}$ such that $p\leq q$. 
By Lemma \ref{norms-of-restrictions} and Definition \ref{def-A-xi-G}
we have $\|A|[\alpha,\omega_1)\|=\|i_{p,q}(B)|[\alpha,\omega_1)\|$.
By the fact that $p\in \FF_{X_q, \alpha}$  we have that
$$\|A|\{\alpha\}\|\geq\|i_{p,q}(B)|\{\alpha\}\|\geq\|i_{p,q}(B)|[\alpha,\omega_1)\|
=\|A|[\alpha,\omega_1)\|.$$
This shows that $\pi_\alpha$ is an isometry when restricted to
$\bigcup_{q\in \GG}\A^\GG_{X_q}|[\alpha,\omega_1)$ which is dense in $\A^\GG$ by Lemma \ref{inductive-limit},
and so the representation is faithful.
\end{proof}

\begin{proposition}\label{thintallfully} 
Suppose that $\GG\subseteq \PP$ is an $\FF$-rich family. Then 
$\A^\GG$ is a scattered thin-tall fully noncommutative C*-algebra
such that 
\begin{enumerate}
\item $\I^{At}_\alpha(\A^\GG)=\A_\alpha^\GG$, 
\item there is a $*$-isomorphism $j_\alpha: 
 \A^\GG/\I^{At}_\alpha(\A^\GG) \rightarrow \A^\GG|[\alpha, \omega_1)$
satisfying 
$$j_\alpha([A]_{\I^{At}_\alpha(\A)})=A|[\alpha, \omega_1),$$
\item the collection $\{[A_{\alpha, m, n}]_{\I^{At}_\alpha(\A)}: n, m\in \N\}$ satisfies the
matrix units relations and generates the essential ideal $At(\A^\GG/\I^{At}_\alpha(\A^\GG))$.
\end{enumerate}
\end{proposition}
\begin{proof} By Theorem 1.4 of \cite{cantor-bendixson}
it is enough to prove (1) - (3) to conclude that $\A$ is a scattered thin-tall fully noncommutative 
C*-algebra.

The proof of (1) - (3) is by induction on $\alpha<\omega_1$. 
For $\alpha=0$ we have that $\I_\alpha=\{0\}$ and  so (1)  and (2) are trivial. Also
$A^\GG_{0, n, m}=1_{0, n, m}$ by Definition \ref{def-P}, so these elements 
satisfy the matrix unit relations. Moreover they generate the algebra of all
compact operators on $\ell_2(\{0\}\times\N)$ which is an essential ideal
in $\B_{\{0\}\times\N}$. Since $\pi_0$ from Lemma \ref{quotientembedding}
is faithfull,  the collection $\{A_{0, m, n}: n, m\in \N\}$ 
generates an essential ideal isomorphic to an algebra of all compact operators 
on a Hilbert space, so by Theorem 1.2 (4) of \cite{cantor-bendixson} this
ideal is $\I^{At}(A^\GG)$ as required.

Now suppose we are done for
$\beta<\alpha<\omega_1$. 

(1)
If $\alpha$ is a limit ordinal, then by 1.4 of \cite{cantor-bendixson} and the inductive hypothesis we have
$$I_\alpha^{At}(\A)={\overline{\bigcup_{\beta<\alpha}I_\beta^{At}(\A)}}=
{\overline{\bigcup_{\beta<\alpha}\A_\beta}}=\A_\alpha.$$
If $\alpha=\beta+1$, then (3) of the inductive hypothesis implies (1).

(2) follows from Lemma \ref{quotientembedding}.

(3)  Is proved like in the case $\alpha=0$.

\end{proof}

\section{An operator algebra along a  construction scheme}

In this section we
adopt the terminology and the notation of 
Section 5. We will use the constructions scheme of \cite{stevo-scheme} described in Section 2.2
to build appropriate $\FF$-rich families $\GG$ in the partial order $\PP$ of approximations whose
inductive limit $\A^\GG$ will have interesting properties described in the introduction.
To prove the main theorem of this section we need one more general lemma:

\begin{lemma}\label{af-projections}
Suppose that $\A$ is an AF C*-algebra where $\{\A_D: D\in \mathcal D\}$ is
 a directed family of finite-dimensional subalgebras with dense union. Let
$P\in \A$ be a projection. Then for every $0<\varepsilon<1$ there is $D\in \mathcal D$ and a  projection
$Q\in \A_D$ such that $\|Q-P\|<\varepsilon$.
\end{lemma}
\begin{proof}
Let $D\in \mathcal D$ be such that there is $A\in A_D$ satisfying $\|A-P\|<\varepsilon/6$.
By considering $(A+A^*)/2$ instead of $A$ we may assume that $A$ is self-adjoint and $\|A-P\|<\varepsilon/6$.
As $\A_D$ is finite dimensional, it is $*$-isomorphic to the direct sum of full matrix
algebras. Let $\pi$ be the isomorphism.
The matrix $\pi(A)$ is self-adjoint, so it can be diagonalized. As $\|A-P\|<\varepsilon/6$ we have
 that $\|A^2-A\|<\varepsilon/2$ and so the distance of each entry 
 on the diagonal  of the diagonalized $\pi(A)$ from $0$ or $1$ cannot 
 be bigger than $\varepsilon/2$, so
 there is a  projection  $Q\in \A_D$ such that $\|\pi(Q)-\pi(A)\|<\varepsilon/2$ and hence
 $\|Q-A\|<\varepsilon/2$ and
 $\|Q-P\|<\varepsilon$ as required. 
\end{proof}

\begin{theorem}\label{theorem-main} Suppose that there exists a construction scheme $\F$
with  allowed parameters $(r_k)_{k\in \N}$ and $(n_k)_{k\in \N}$, where
$n_k=3$ for each $k\in \N\setminus\{0\}$  and a partition $(P_m)_{m\in \N}$
of $\N$ into  infinite sets  such that for every $m\in \N$ and every uncountable
$\Delta$-system $T$ of finite subsets of $\omega_1$ there exist $F\in \F$ of arbitrarily large rank in $P_m$
which fully captures a subsystem of $T$.

Then there is an $\FF$-rich family $\GG$ of elements of $\PP$ such that the scattered thin-tall fully
noncommutative C*-algebra 
$\A^\GG$ has the following properties:
\begin{enumerate}
\item There is a  nondecreasing unbounded sequence $(l_k)_{k\in \N}\subseteq \N$ and a
directed family of finite dimensional algebras $\{A^\GG_X: X=F\times [0, l_k),
F\in \F_k, k\in \N\}$
whose union $\B$ is dense in $\A$ such that whenever
$(P_\xi: \xi<\omega_1)\subseteq \B$ is a family of projections which generate a nonseparable
subalgebra of $\A^\GG$, then for every $\varepsilon>0$
\begin{enumerate}
\item there are $\xi_1<\xi_2<\xi_3<\omega_1$ such that
 $\|P_{\xi_1}-P_{\xi_2}P_{\xi_3}\|<\varepsilon$,
\item there are $\xi_1<\xi_2<\omega_1$ such that 
$\|[P_{\xi_1}, P_{\xi_2}\|<\varepsilon$,
\item   there are 
$\xi_1<\xi_2<\omega_1$ such that 
$\|[P_{\xi_1}, P_{\xi_2}]\|>1/2-\varepsilon.$
\end{enumerate}
\item $\A^\GG$ has no uncountable irredundant subset,
\item  $\A^\GG$ has no nonseparable abelian subalgebra.
\end{enumerate}
\end{theorem}
\begin{proof}

Fix an enumeration $((v^m, w^m): m\geq 3)$, with possible repetitions, of all pairs
of orthogonal complex vectors with finitely many coordinates, all of them 
rational,  such that $v^m, w^m\in \C^m$ for each
$m\geq 3$ (we abuse notation and identify $\C^{m'}$ with a subset of $\C^m$
for $m'\leq m$).

We construct the sequence $(l_k)_{k\in\in \N}\subseteq \N$
and $\GG=\{p_F: F\in \F\}\subseteq \PP$ by induction with respect to $k\in \N$
such that $F\in \F_k$. Moreover, for each $k\in\N$ we require  that whenever 
$F, F'\in \F_k$ are such that $F\setminus F'<F'\setminus F$ (cf. Definition \ref{def-scheme} (2)), then
\begin{enumerate}[(*)]
\item[(*)] $p_F$ and $p_{F'}$ are in the convenient position as witnessed by $\phi_{F', F}$.
\item[(**)] $a_{p_F}=F$,
\item[(***)] $n_\xi^{p_F}=l_k$  for all $\xi\in F$ and $F\in \F_k$ and $k\in \N$.
\end{enumerate}

For $k=0$ we have that $\F_1=[\omega_1]^1$ by Definition \ref{def-scheme} (1), so we define
$p_F$ for $F=\{\xi\}$ to be the element of $\PP$ such that
\begin{itemize}
\item $a_{p_F}=\{\xi\}$,
\item $n_{p_F}^\xi=l_0=1$,
\item $A^{p_F}_{\xi, 0,0}=1_{\xi, 0, 0}$.
\end{itemize}

Suppose that we have constructed $p_{F}$s for all $F\in F_{k'}$ for $k'\leq k$
satisfying (*) - (***). Now we need to define
the $p_F$s for $F\in \F_{k+1}$. Since $n_{k+1}=3$, each $F\in \F_{k+1}$ is the union of 
the maximal elements $G_1, G_2, G_3$ of $\F|F$ which form an
increasing $\Delta$-system by Definition \ref{def-scheme} (3).
If $k\in P_1$, then we define $p_F$ as the amalgamation of $p_{G_1}, p_{G_2}, p_{G_3}$
of type 1 from Lemma \ref{amaltype1}. If $k\in P_2$, then we define $p_F$ as the amalgamation
of $p_{G_1}, p_{G_2}, p_{G_3}$ of type 2 from Lemma \ref{amaltype2}. If $k\in P_m$ for $m\geq 3$, 
and $l_k<m$, then we define $p_F$ as the amalgamation of $p_{G_1}, p_{G_2}, p_{G_3}$
of type 1 from Lemma \ref{amaltype1}. If $k\in P_m$ for $m\geq 3$, 
and $l_k\geq m$, 
then we define $p_F$ as the amalgamation of $p_{G_1}, p_{G_2}, p_{G_3}$ of type 3 
for vectors $(v^m, w^m)$ from Lemma \ref{amaltype3}.

Observe that amalgamation of type 1  increases $l_k$, so $l_k\rightarrow\infty$ when
$k\rightarrow\infty$.

First let us note that our inductive hypothesis (*) - (***) is preserved
when we pass from $k\in \N$ to $k+1$. Let $F, F'\in \F_{k+1}$ be such that
$F\setminus F'<F'\setminus F$. By Definition \ref{def-scheme}
(2) there is an order preserving bijection
$\phi_{F', F}: F\rightarrow F'$ and $F\cap F'<F\setminus F'<F'\setminus F$.
In particular  Definition \ref{def-scheme} (2) implies that
the maximal elements of $\F|F$ are sent by $\phi_{F,', F}$ onto the maximal
elements of $\F|F'$, on the other hand (3) of  \ref{def-scheme} 
implies that these maximal elements
form the canonical decomposition consisting of elements in $\F_{k}$, which
in fact are used in the construction of $p_F$ or $p_{F'}$.
Now to verify Definition \ref{convenient} in order to check (*) we note that the
amalgamations described in Lemmas \ref{amaltype1}, 
\ref{amaltype2}, \ref{amaltype3}  consist of constructions
of operators which depend only on the place of the involved objects in $F$, so
Definition \ref{convenient} and (*) are satisfied for  $p_F$ and $p_{F'}$.
 (**) and (***)
follow from the descriptions of the amalgamations from Lemmas \ref{amaltype1}, 
\ref{amaltype2}, \ref{amaltype3}. We have $l_{k+1}=l_k$ if $k\in \N\setminus P_1$
and $l_{k+1}>l_k$ if $k\in P_1$ and Definition \ref{def-scheme} guarantees that 
the amalgamations which follow Lemma \ref{amaltype1} can be done in ``the same
way" up to
the bijection $\phi_{F', F}$
 and so obtaining $n^\xi_{p_F}=l_{k+1}=n^{\xi'}_{p_{F'}}$ for any $F, F'\in \F_{k+1}$
 and $\xi\in F$ and  $\xi'\in F'$.
This completes the construction of $\GG=\{p_F: F\in \F\}$ and
 determines completely the C*-algebra
$\A^\GG$ as in Definition  \ref{def-A-G}.

Now note that $\GG$ is $\FF$-rich
as in Definition \ref{def-F-rich}.
First note that $p_{F'}\leq p_{F}$ whenever $F'\subseteq F$ and $F, F'\in \F$.
This can be proved by induction on $k\in \N$ such that $F\in \F_k$.
Note that it is true if $F'$ is a maximal element of $F|\F$, because then
$F'$ is in the canonical decomposition of $F$ by Definition \ref{def-scheme} (3)
and we use $p_{F'}$ in the construction of $p_F$ obtaining
$p_{F'}\leq p_{F}$ by the Lemmas \ref{amaltype1}, 
\ref{amaltype2}, \ref{amaltype3}.  
Now we proceed with the inductive argument, given $F'\subsetneq F$  either $F'$
is below a maximal element $G$ of $\F|F$ or it is one of the maximal elements.
The latter case is proved above and the former follows from the inductive assumption
for the pair $F', G$ and from the transitivity of the order in $\PP$.

To prove the directedness of $\GG$ take $F, F'\in \F$ and use  
the cofinality of  $\F$  in $[\omega_1]^{<\omega}$
(Definition \ref{def-scheme}) to find $F''\in \F$ such taht $F\cup F'\subseteq F''$.
By the above arguments we have $p_F, p_{F'}\leq p_{F''}$.

Now let $X=a\times[0, l)\in [\omega_1\times\N]^{<\omega}$ and $\alpha\in \omega_1$
and aim at proving further parts of the $\FF$-richness.
Consider the $\Delta$-system 
$T=\{a\cup\{\alpha, \xi\}: \max(a\cup\{\alpha\})<\xi<\omega_1\}$ of finite subsets of $\omega_1$. By the hypothesis
there is $k\in P_1$ with $l_k\geq l$ and $F\in \F$ such that $F$
fully captures a subsystem of  $T$. In
particular $F=G_1\cup G_2\cup G_3$ for some $G_1, G_2, G_3\in \F_k$ and
$X\subseteq X_{p_{G_1}}$ and $\alpha\in a_{p_{G_1}}$. 
By the construction, we do the amalgamation of type
 1 like in Lemma \ref{amaltype1} while constructing $p_F$ and
so $p_F$ is in $\FF_{X_{p_{G_1}}, \alpha}$ but this implies that it is in 
$\FF_{X, \alpha}$ as required for $\FF$-richness  in Definition \ref{def-F-rich}.

Proposition \ref{thintallfully} implies that $\A^\GG$ as in Definition \ref{def-A-G}
is a thin-tall fully noncommutative
scattered C*-algebra. 

To prove (1) the directed family of finite dimensional subalgebras of
$\A^\GG$ is $\{\A^\GG_{X_p}: p\in \GG\}$ as in Definition \ref{def-A-G}. 
By Lemma \ref{restriction1} the algebras $\A^\GG_{X_p}$ are *-isomorphic to
the algebras $\A^p_{X_p}$ and they are finite dimensional since they are equal to
 $\B_{X_p}$ by Lemma \ref{generating-B}. Let $\B=\bigcup\{\A^\GG_{X_p}: p\in \GG\}$.

Suppose that $\{P_\xi: \xi<\omega_1\}\subseteq \B$ is a collection of projections which
generate a nonseparable subalgebra of $\A^\GG$. So, there must be distinct
$\alpha_\xi\in \omega_1$ such that $P_\xi|\big(\{\alpha_\xi\}\times\N\big)\not=0$. Since
$\B_{\{\alpha_\xi\}\times\N}$ is invariant for $\A^\GG$ it follows that
$P_\xi|\big(\{\alpha_\xi\}\times\N\big)$ is a non-zero projection. Moreover it
is not the unit of $\B_{\{\alpha_\xi\}\times\N}$ because such a  unit would produce
a unit of $\A^\GG/\I^{At}_{\alpha_\xi}(\A^\GG)$ by Lemma \ref{quotientembedding} and Theorem
\ref{thintallfully}, which is impossible because $\A^\GG$ is the union of proper ideals
$\I^{At}_{\alpha}(\A^\GG)$ for $\alpha<\omega_1$.

Let $F_\xi\in \F$ be
such that $\alpha_\xi\in F_\xi$, 
$P_\xi\in \A^\GG_{X_{p_{F_\xi}}}=\A^\GG_{F_\xi\times[0,l_\xi)}$
for each $\xi\in \omega_1$, where $l_\xi=l_{k}$ for $F_\xi\in \F_k$
  and $P_\xi|\big(\{\alpha_\xi\}\times[0, l_\xi)\big)$ is a nonzero
  projection which is not the unit of $\B_{\{\alpha_\xi\}\times[0, l_\xi)}$.
  This can be obtained from the cofinality of $\F$ and the fact that 
  $l_k\rightarrow\infty$ when $k\rightarrow \infty$.

Let $Q_\xi\in \A^{p_{F_\xi}}_{F_\xi\times[0,l_\xi)}$ be such that
$i_{\GG, p_{F_\xi}}(Q_\xi)=P_\xi$. Note that by Lemma \ref{restriction1}
$Q_\xi$s are projections and  $Q_\xi|\big(\{\alpha_\xi\}\times[0, l_\xi)\big)$ is a nonzero
  projection which is not the unit of $\B_{\{\alpha_\xi\}\times[0, l_\xi)}$ for
  each $\xi<\omega_1$.
  
By passing to an uncountable subset,
we may assume that  $T=\{F_\xi: \xi<\omega_1\}$
 forms an increasing $\Delta$-system of elements of $\F_{k'}$ for a fixed $k'\in\N$  and that 
{$$|\langle Q_\xi(e_{\eta,l}), e_{\eta, l'}\rangle -
 \langle Q_{\xi'}(e_{\phi_{F_{\xi'}, F_\xi}(\eta),l}), e_{\phi_{F_{\xi'}, F_\xi}(\eta), l'}\rangle|
 <\varepsilon/2l_{k'}$$}
for every  $(\eta,l), (\eta,l')\in F_\xi\times[0,l_{k'})$  and every $\xi<\xi'<\omega_1$.
This guarantees that
$$\|j_{\phi_{F_{\xi}, F_{\xi'}}}(Q_{\xi})-Q_{\xi'}\|<\varepsilon/2
\leqno (+)$$
 for every $\xi<\xi'<\omega_1$.
 Now let us prove item (a) of (1). 
 By the hypothesis on $\F$ there is $k\in P_2$ bigger than $k'$	
 and $F\in \F_{k+1}$ which fully captures $T$, i.e. the canonical decomposition of $F$ is
 $\{G_1, G_2, G_3\}$ and there are $\xi_1<\xi_2<\xi_3<\omega_1$ such that $F_{\xi_i}\subseteq G_i$
 and $\phi_{G_j, G_i}[F_{\xi_i}]=F_{\xi_j}$ for all $1\leq i, j\leq 3$. As $\phi_{G_j, G_i}$
 are order preserving, they must agree with $\phi_{F_{\xi_j}, F_{\xi_i}}$ on $F_{\xi_i}$, so
 (+) implies that 
{$$\|j_{\phi_{G_3, G_i}}(Q_{\xi_3})
 -Q_{\xi_i}\|<\varepsilon/2$$}
holds for  $i=1, 2$. Since we use amalgamation
of type 2 at the construction of $p_F$ for $k\in P_2$ by Lemma \ref{amaltype2} we have
{$$i_{p_F, p_{G_3}}(Q_{\xi_3})
i_{p_F, p_{G_2}}(j_{\phi_{G_3, G_2}}(Q_{\xi_3}))=
i_{p_F, p_{G_1}}(j_{\phi_{G_3, G_1}}(Q_{\xi_3}))^2,$$}
and so
{$$\|i_{p_F, p_{G_3}}(Q_{\xi_3})
i_{p_F, p_{G_2}}(Q_{\xi_2})-
i_{p_F, p_{G_1}}(Q_{\xi_1})^2\|<\varepsilon$$}
and hence $\|P_{\xi_3}P_{\xi_2}-P_{\xi_1}\|<\varepsilon$ since 
{$$i_{\GG,p_F}\circ i_{p_F, p_{G_{\xi_i}}}(Q_{\xi_i})=i_{\GG, p_{G_{\xi_i}}}(Q_{\xi_i}) =
i_{\GG, p_{G_{\xi_i}}}(i_{p_{G_{\xi_i}},p_{F_{\xi_i}}}(Q_{\xi_i})) = i_{\GG, p_{F_{\xi_i}}}(Q_{\xi_i})=P_{\xi_i}$$}
by Definition \ref{def-A-xi-G}
and Lemma \ref{restriction1}. This completes the proof of (a) of (1). Item (b) follows from
(a) for $\varepsilon/2$ and by taking the adjoints.

Now let us prove item (c) of (1). For $\xi<\omega_1$ let 
$Q_\xi'\in \B_{F_\xi\times[0,l_{k'})}$ be such  projections that
$\|Q_\xi-Q_\xi'\|<\varepsilon/8$ and there is on orthonormal basis
in $\B_{\{\alpha_\xi\}\times[0, l_{k'})}$ of eigenvectors
for $Q_{\xi}'$ consisting only of vectors with all rational coordinates with respect
to our canonical basis $(e_{\alpha_\xi, l}: 0\leq l< l_{k'})$. 
Note that by Lemma \ref{generating-B} we have that $Q_\xi'\in \A^{p_{F_\xi}}_{X_{p_{F_\xi}}}$.
Since
 $Q_\xi|\big(\{\alpha_\xi\}\times[0, l_{k'})\big)$ is a nonzero
  projection which is not the unit of $\B_{\{\alpha_\xi\}\times[0, l_{k'})}$ for
  each $\xi<\omega_1$, $Q_\xi'$
  may be assumed to have the same rank as $Q_\xi$ and so there are orthogonal unit vectors $v^\xi, w^\xi\in \C^{l_{k'}}$ with all rational coordinates
  such that
$$Q_\xi'(\sum_{l<l_{k'}}v^\xi_le_{\alpha_\xi,l})=\sum_{l<l_{k'}}v^\xi_le_{\alpha_\xi,l},\ \ 
Q_\xi'(\sum_{l<l_{k'}}w^\xi_le_{\alpha_\xi,l})=0.$$
As there are only countably many such vectors we may assume that all of them are equal
to a pair $(v, w)$ and moreover that
$$\|j_{\phi_{F_{\xi}, F_{\xi'}}}(Q_{\xi}')-Q_{\xi'}'\|<\varepsilon/4
\leqno (++)$$
 for every $\xi<\xi'<\omega_1$.

By the hypothesis on $\F$ there is $k\in P_m$ bigger than $k'$ such that
$v_m=v=(v_1, ... v_{l_{k'}})$ and $w_m=w= (w_1, ... w_{l_{k'}})$ 
 and there is $F\in \F_{k+1}$ which fully captures $T$, i.e., 
 the canonical decomposition of $F$ is
 $\{G_1, G_2, G_3\}$ 
  and there are $\xi_1<\xi_2<\xi_3<\omega_1$ such that $F_{\xi_i}\subseteq G_i$
 and $\phi_{G_j, G_i}[F_{\xi_i}]=F_{\xi_j}$ for all $1\leq i, j\leq 3$. 
 Note that
 $\alpha_{\xi_i}$s are not in the root of $\{G_1, G_2, G_3\}$ as they are not in
 the root of $F_\xi$s.
 As $\phi_{G_j, G_i}$
 are order preserving, they must agree with $\phi_{F_{\xi_j}, F_{\xi_i}}$ on $F_{\xi_i}$, so
 (++) implies that 
 {$$\|j_{\phi_{G_1, G_i}}(Q_{\xi_1}')
 -Q_{\xi_i}'\|<\varepsilon/4$$}
holds for  $i= 2, 3$. Since we use amalgamation
of type 3 at the construction of $p_F$ for $k\in P_m$ by Lemma \ref{amaltype3} we have
 {$$\|[i_{p_F, p_{G_1}}(Q_{\xi_1}'), 
i_{p_F, p_{G_2}}(j_{\phi_{G_1, G_2}}(Q_{\xi_1}'))]\|=1/2$$}
and so
 {$$\|[i_{p_F, p_{G_1}}(Q_{\xi_1}'), 
i_{p_F, p_{G_2}}(Q_{\xi_2}')]\|\geq1/2-\varepsilon/2$$}
and hence 
{$$\|[i_{p_F, p_{G_1}}(Q_{\xi_1}), 
i_{p_F, p_{G_2}}(Q_{\xi_2})]\|\geq1/2-\varepsilon$$}
as $\|Q_\xi-Q_{\xi'}'\|<\varepsilon/8$ for each $\xi<\xi'<\omega_1$, and finally
$$\|[P_{\xi_1}, P_{\xi_2}]\|\geq1/2-\varepsilon$$
 since 
{$$i_{\GG,p_F}\circ i_{p_F, p_{G_{\xi_i}}}(Q_{\xi_i})=i_{\GG, p_{G_{\xi_i}}}(Q_{\xi_i}) =
i_{\GG, p_{G_{\xi_i}}}(i_{p_{G_{\xi_i}},p_{F_{\xi_i}}}(Q_{\xi_i})) = i_{\GG, p_{F_{\xi_i}}}(Q_{\xi_i})=P_{\xi_i}$$}
by Definition \ref{def-A-xi-G}
and Lemma \ref{restriction1}. This completes the proof of (c) of (1).

The proof of (2) will be based on (1) (a) and Lemma \ref{af-projections}. Suppose that $\A^\GG$ contains 
an uncountable irredundant set $\{Q_\xi: \xi<\omega_1\}$.
By Lemma \ref{irr-projections} we may assume that all $Q_\xi$s are projections. 
For each $\xi$ let $\A_{\omega_1\setminus\{\xi\}}$
be the C*-subalgebra of $\A^\GG$
generated by the set $\{Q_\eta: \eta\in\omega_1\setminus\{\xi\}\}$.
By passing to an uncountable subset we may assume that there is $\varepsilon>0$ such that for
each $\xi<\omega_1$ we have $\|A-Q_\xi\|\geq\varepsilon$ for each 
$A\in \A_{\omega_1\setminus\{\xi\}}$. 
Let $P_\xi\in \B$ be a projection satisfying $\|P_\xi-Q_\xi\|<\varepsilon/4$ which is obtained using
Lemma \ref{af-projections}. By (1) (a) there are $\xi_1<\xi_2<\xi_3<\omega_1$ such that
$\|P_{\xi_1}-P_{\xi_2}P_{\xi_3}\|<\varepsilon/4$.
 This implies that $\|Q_{\xi_1}-Q_{\xi_2}Q_{\xi_3}\|<\varepsilon$
which contradicts the defining property of $\varepsilon$ and completes the proof of (2). 

The proof of (3) will be based on (1) (c) and Lemma \ref{af-projections}. Suppose that $\A^\GG$ contains a
nonseparable abelian subalgebra. As subalgebras of scattered algebras are scattered,
 and scattered locally compact spaces are totally disconnected, it follows that $\A^\GG$  contains an 
 uncountable Boolean algebra of (commuting) projections $\{Q_\xi: \xi<\omega_1\}$.
 In particular $\|Q_\xi-Q_{\xi'}\|=1$ for all $\xi<\xi'<\omega_1$.

Let $P_\xi\in \B$ for $\xi<\omega_1$ be  projections satisfying 
$\|P_\xi-Q_\xi\|<1/10$ for each $\xi<\omega_1$ which is obtained using
Lemma \ref{af-projections}. In particular 
{ $\|P_\xi-P_{\xi'}\|\geq 8/10$}
for all $\xi<\xi'<\omega_1$ and so they generate a nonseparable C*-algebra.

We have $\|P_{\xi_1}P_{\xi_2}-Q_{\xi_1}Q_{\xi_2}\|<1/5$
and $\|P_{\xi_2}P_{\xi_1}-Q_{\xi_2}Q_{\xi_1}\|<1/5$ for each $\xi_1<\xi_2<\omega_1$, so
$[P_{\xi_1},P_{\xi_2}]< 2/5$ for each $\xi_1<\xi_2<\omega_1$.
But by (1) (c) there are $\xi_1<\xi_2<\omega_1$ such that
$\|[P_{\xi_1},P_{\xi_2}]\|\geq 2/5$, a contradiction.

\end{proof}

\bibliographystyle{amsplain}

\end{document}